\documentclass[12pt,times]{article}
\title{Nonlocal flocking dynamics: Learning the fractional order of PDEs from particle simulations}

\author{Zhiping Mao$^{1}$,
        Zhen Li$^{1}$\footnote{Corresponding author (\href{mailto:zhen_li@brown.edu}{zhen\_li@brown.edu})}~,
        and George Em Karniadakis$^{1}$ \\
\small{$^{1}$ Division of Applied Mathematics, Brown University, Providence, Rhode Island, 02912, USA} }

\date{\today}

\RequirePackage{lineno}
\usepackage[colorlinks=true]{hyperref}
\usepackage{graphicx}
\usepackage{caption}
\usepackage{subcaption}
\usepackage{multirow,array,color,mathrsfs,subcaption}

\usepackage{amsmath,fancyhdr,amsthm,amssymb,amsfonts,version}
\usepackage{bbm,version}
\usepackage[section]{placeins}
\usepackage{stmaryrd}
\usepackage{bm}

\DeclareGraphicsExtensions{.eps,.mps,.pdf,.jpg,.png}
\graphicspath{{figures/}{../figures/}}

\usepackage{dsfont}

\usepackage[top=1.5cm,bottom=2.5cm,left=1.5cm,right=1.5cm]{geometry}

\usepackage[figuresright]{rotating}
\usepackage{extarrows}

\newtheorem{thm}{Theorem}

\newtheorem{exam}{Example}

\begin{document}

\maketitle
\vspace{-20pt}
\begin{abstract}
Flocking refers to collective behavior of a large number of interacting entities, where the interactions between discrete individuals produce collective motion on the large scale.
We employ an agent-based model to describe the microscopic dynamics of each individual in a flock, and use a fractional partial differential equation (fPDE) to model the evolution of macroscopic
quantities of interest. The macroscopic models with phenomenological interaction functions are derived by applying the continuum hypothesis to the microscopic model. Instead of specifying the fPDEs with an \emph{ad hoc} fractional order for nonlocal flocking dynamics, we learn the effective nonlocal influence function in fPDEs directly from particle trajectories generated by the agent-based simulations. We demonstrate how the learning framework is used to connect the discrete agent-based model to the continuum fPDEs in one- and two-dimensional nonlocal flocking dynamics. In particular, a Cucker-Smale particle model is employed to describe the microscale dynamics of each individual, while Euler equations with non-local interaction terms are used to compute the evolution of macroscale quantities. The trajectories generated by the particle simulations mimic the field data of tracking logs that can be obtained experimentally. They can be used to learn the fractional order of the influence function using a Gaussian process regression model implemented with the Bayesian optimization. We show in one- and two-dimensional benchmarks that the numerical solution of the learned Euler equations solved by the finite volume scheme can yield correct density distributions consistent with the collective behavior of the agent-based system solved by the particle method. The proposed method offers new insights on how to scale the discrete agent-based models to the continuum-based PDE models, and could serve as a paradigm on extracting effective governing equations for nonlocal flocking dynamics directly from particle trajectories.
\end{abstract}


\section{Introduction}
Collective behavior is a widespread phenomenon in physical and biological systems, and also in social dynamics, such as collective motion of self-propelled particles~\cite{2018Nagai}, bird flocking~\cite{antoniou2013congestion}, fish schooling~\cite{niwa1994self}, swarms of insects~\cite{theraulaz1995modelling}, trails of foraging ants~\cite{beekman2001phase}, herds of mammals~\cite{okubo1986dynamical}, complex networks~\cite{mahmoodi2017self}, etc. Although this collective dynamics are at very different scales and levels of complexity, the mechanism of self-organization, where local interactions for the individuals lead to a coherent group motion, is very general and transcends the detailed objects~\cite{2008Giardina}.
To this end, simulation and modeling of both physical~\cite{2001Levine} and biological~\cite{2018Bernardi} systems have driven a rich field of research to explore how individual behavior engenders large scale collective motion. There are generally two different approaches to investigate the underlying mechanics: 1) at the microscopic level, agent-based models are developed to simulate dynamics of each individual in flocks, such as swarms, tori, and polarized groups; 2) at the macroscopic level, the mathematical modeling approach is based on continuum models described by partial differential equations (PDEs).\
Agent-based models assume behavioral rules at the level of the individual, whose microscopic dynamics is governed by an evolution equation affected by the social forces, including alignment of velocities, attraction, and short-range repulsion, acting on it~\cite{2008Giardina}.
Because the number of agents in a coupled dynamics is often large, the agent-based models cannot be solved exactly but can be easily implemented using numerical simulations. Agent-based simulation captures the detailed dynamics of each individual and can handle the increasing complexity of realworld flocking systems. However, for a flocking dynamics involving a large number of agents, the agent-based model becomes computationally expensive~\cite{2013Jaffry}.
Alternatively, by assuming that a flocking group is already formed and considering a large enough number of agents to make the mean field approximation meaningful, Eulerian models can be derived by applying the continuum hypothesis to the microscopic dynamics, leading to PDEs for the macroscopic quantities, i.e., mean velocity and population density~\cite{2008Giardina}.

Advances in digital imaging~\cite{2008Ballerini} and high-resolution lightweight GPS devices~\cite{2010Nagy} allow gathering long-time and long-distance trajectories of individuals in flocks. For example,
Ballerini et al.~\cite{2008Ballerini} used stereometric and computer vision techniques to measure 3D individual birds positions in huge starling flocks (up to $2\,600$ European starlings) in the field.
Nagy et al.~\cite{2010Nagy} employed high-resolution lightweight GPS devices to collect track logs of homing pigeons flying in flocks of up to 13 individuals and analyzed the hierarchical group dynamics in the pigeon flocks.
Lukeman et al.~\cite{2010Lukeman} recorded time series of 2D swimming flocks (up to 200 ducks) by oblique overhead photography and analyzed each individual's position, velocity and trajectory.
Tunstrom et al.~\cite{2013Tunstrom} used automated tracking software for fish to obtain detailed data regarding the individual positions and velocities of schooling fish (up to 300 fish) over long periods of time.
Despite significant development of these experimental ideas, a gap between flocking theory/modeling and experiment still exists. Due to the rich diversity of theoretical models with distinct forms of interaction~\cite{2010Lukeman}, it is necessary to assess which of them correspond to actual behavior in specified collective motions in nature. Numerical simulations alone cannot tackle this problem because flocking patterns similar to experimental observations can be generated by significantly different model mechanisms. Although some mechanisms through which collective motion is achieved have been qualitatively understood by simulation and modeling, a quantitative test of the model assumptions in realistic data is still challenging. Therefore, it is of great interest to establish a direct connection between the individual trajectories obtained in experiments and the effective governing equations in the mathematical approach.

In the present work, we use agent-based simulations for self-organized dynamics to generate particle trajectories of individuals in flocks to mimic field data that can be obtained by computer vision techniques~\cite{2008Ballerini} or high-resolution GPS data of flocks~\cite{2016Chen}. We also use continuum models formulated in PDEs to capture the evolution of macroscopic quantities.
It has been reported that the mathematical models using local rules do not fully describe the social interactions between individuals in flocks~\cite{2007Eftimie}. In fact, animals in flocks generally do not interact mechanically, rather they are influenced by other individuals a certain distance away~\cite{2010Cristiani}. Animal communications can result in nonlocal interactions~\cite{2013Giuggioli}, which should be modeled by nonlocal mathematical models.
For nonlocal flocking dynamics, instead of specifying the phenomenological fractional partial differential equations (fPDEs) with an empirical fractional order, we use a Gaussian process regression model implemented with the Bayesian optimization to learn the effective nonlocal influence function for fPDEs directly from particle trajectories, as shown in Fig.~\ref{fig:sketch}. Both one- and two-dimensional nonlocal flocking dynamics are performed to demonstrate the procedure how the proposed learning-framework is used to connect the discrete agent-based model and the continuum-based fPDEs model.

\begin{figure}
  \centering
  \includegraphics[width=0.68\textwidth]{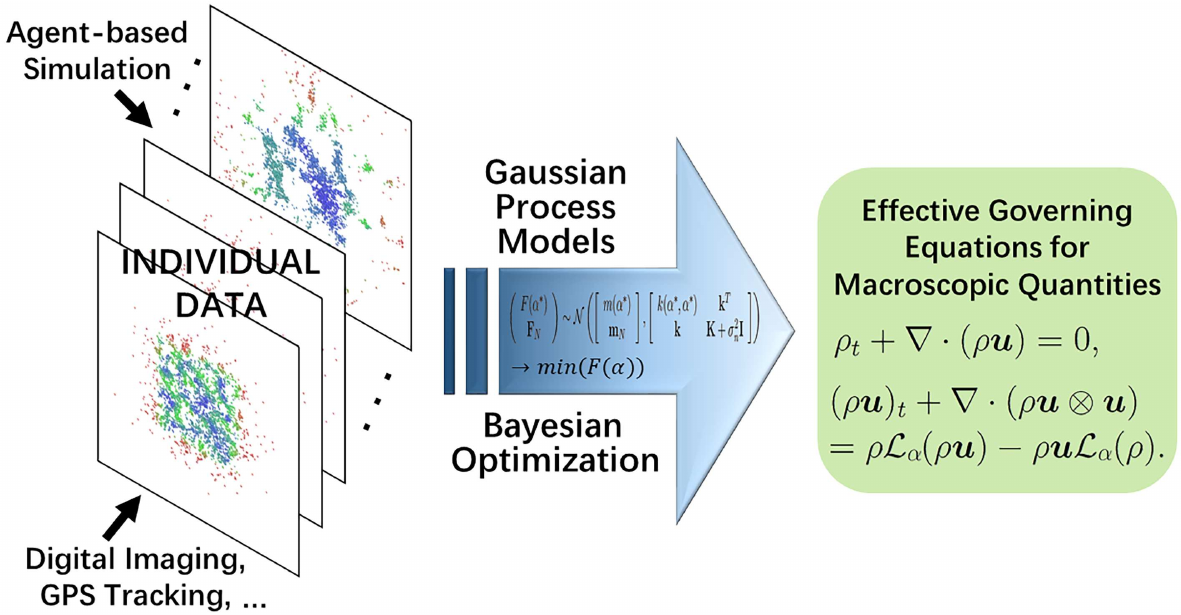}\\
  \caption{Schematic of learning the fractional order of governing PDEs of nonlocal flocking dynamics from time series of particle trajectories, which can be either obtained from experiments, i.e., digital imaging and GPS tracking, or by agent-based simulations. A Gaussian process model implemented with the Bayesian optimization is employed to minimize the loss and to extract the effective nonlocal influence function.}\label{fig:sketch}
\end{figure}

For the agent-based system, we adopt the Cucker-Smale model with a nonlocal alignment term developed in~\cite{2007Cucker}. Using a singular kernel makes the nonlocal alignment term to be a nonlinear function involving fractional Laplacian for the Euler equations which has \emph{mass and momentum conservations}. The choice of the fractional order of the fractional Laplacian offers a flexibility in modeling the complicated interactions between agents.
It is worth noting that the computation of the corresponding Euler equations with nonlocal forces is very challenging. The challenges are two-fold due to the fact that the alignment term is not only \emph{nonlocal} but also \emph{nonlinear}. A typical challenge for nonlocal problems is the high computation cost of matrix-vector multiplication, which is usually $O(K^2)$ with $K$ being the total number of degrees of freedom at each time step when using a explicit time scheme. We resolve this issue by using the fast Fourier transform, resulting in a computation cost reduced to $O(K\log K)$. For the issue of nonlinearity, we use a piecewise constant approximation, i.e., a finite volume scheme, which can significantly simplify the computation of the nonlinear nonlocal term. Moreover, we show that the proposed finite volume scheme \emph{preserves the mass and momentum}, both in one and two space dimensions.

The remainder of this paper is organized as follows. In Section~\ref{sec:22}, we introduce a typical agent-based model and the corresponding Euler equations for flocking dynamics. In Section~\ref{sec:2}, we describe the details of numerical algorithms to solve the agent-based model and to discretize the fPDEs, including the statistical algorithm for sampling macroscopic quantities from discrete particle data, the velocity-Verlet algorithm for time integration, and the finite volume approximation for the Euler system of equations. Subsequently, in Section~\ref{sec:3} we present the numerical examples of one- and two-dimensional nonlocal flocking dynamics, and compare the results obtained by the agent-based simulations and by solving the fPDEs. In Section~\ref{sec:4}, we introduce the learning framework to infer the effective influence function from particle trajectories using a Gaussian process regression model implemented with the Bayesian optimization. Finally, we conclude with a brief summary and discussion in Section~\ref{sec:5}.

\section{Mathematical models}\label{sec:22}
In this section, we briefly introduce a typical governing equation of the agent-based model, followed by the Euler equations for flocking dynamics, as well as the conservation laws of the Euler system of equations.
\begin{enumerate}
  \item \textbf{Agent-based model}:
The agent's behavior is characterized by its position and velocity. We consider the following agent-based model, i.e., the C-S model with a alignment term, at the microscale:
\begin{equation}\label{eq:Micro}
    \begin{aligned}
      &\dot{\bm{x}}_i = \bm{v}_i, \\
      &\dot{\bm{v}}_i = \frac{1}{N}\sum_{j=1}^N \phi(|\bm{x}_i- \bm{x}_j|) (\bm{v}_j - \bm{v}_i), \quad (\bm{x}_i,\bm{v}_i) \in \mathbb{R}^n \times \mathbb{R}^n,
    \end{aligned}
\end{equation}
%
%\begin{equation}\label{eq:Micro}
%    \left\{
%    \begin{array}{lcc}
%      \dot{\bm{x}}_i = \bm{v}_i, \\
%      \dot{\bm{v}}_i = \frac{1}{N}\sum_{j=1}^N \phi(|\bm{x}_i - \bm{x}_j|) (\bm{v}_j - \bm{v}_i), \quad (\bm{x}_i,\bm{v}_i) \in \mathbb{R}^n \times \mathbb{R}^n,
%    \end{array} \right.
%\end{equation}
%
where $\phi(\cdot)$ is a kernel denoting the influence function.

  \item \textbf{Euler equations}:
For large crowds, i.e., $N\gg 1$, by using the mean field limit argument,  the particle system governed by Eq.~\eqref{eq:Micro} can lead to the following macroscale Euler system of equations:
\begin{equation}\label{Eulern}
\begin{aligned}
  &\rho_t + \nabla \cdot(\rho \bm{u}) = 0,\\
  &\bm{u}_t + \bm{u}\cdot \nabla \bm{u} = [\mathcal{L},\bm{u}](\rho), \; (\bm{x},t ) \in \Omega\times \mathbb{R}^+,
\end{aligned}
\end{equation}
%
%\begin{equation}\label{Eulern}
%    \left\{
%    \begin{array}{lcc}
%      \rho_t + \nabla \cdot(\rho \bm{u}) = 0,\\
%      \bm{u}_t + \bm{u}\cdot \nabla \bm{u} = [\mathcal{L},\bm{u}](\rho), \; (\bm{x},t ) \in \Omega\times \mathbb{R}^+,
%    \end{array} \right.
%\end{equation}
where $\rho$ is the density and $\bm u$ is the velocity of the macro-model, $[\mathcal{L},\bm{u}](\rho): = \mathcal{L}(\rho \bm{u}) - \mathcal{L}(\rho) \bm{u}$ is the commutator forcing and
$$\mathcal{L}(f)(\bm x) := \int_{\mathbb{R}^n} \phi(|\bm x- \bm y|) \left(f(\bm{y}) - f(\bm{x}) \right)d\bm{y}. $$

In this paper, we consider the singular kernel, i.e., $\phi(r) := c_{n,\alpha}|r|^{-(n+\alpha)}$, where $c_{n,\alpha} = \frac{\alpha \Gamma(\frac{n+\alpha}{2})} {2\pi^{\alpha+n/2}\Gamma(1-\alpha/2)}$, which is associated with the action of the \emph{fractional Laplacian} $\mathcal{L}(f) = -(-\Delta)^{\alpha/2}f,\, 0<\alpha<2$, namely,
\begin{equation*}
    \mathcal{L}(f)(\bm x) = \text{p.v.}\; c_{n,\alpha}\int_{\mathbb{R}^n} \frac{f(\bm{y}) - f(\bm{x})}{|\bm x-\bm y|^{n+\alpha}}d\bm{y},
\end{equation*}
where $\text{p.v.}$ means the principle value.
The corresponding forcing is given by the following singular integral
\begin{equation}\label{}
    [\mathcal{L},\bm{u}](\rho) = \text{p.v.}\; c_{n,\alpha}\int_{\mathbb{R}^n} \frac{u(\bm{y}) - u(\bm{x})}{|\bm x-\bm y|^{n+\alpha}}\rho(\bm y)d\bm{y}.
\end{equation}
In the one-dimensional case, the global regularity for $\alpha\in[1,2]$ is proved by Shvydkoy and Tadmor \cite{2017Shvydkoy}. Also, they proved fast velocity alignment as the velocity $u(\cdot,t)$ approaches a constant state, $u\rightarrow \tilde{u}$, with exponentially decaying slope \cite{2017Shvydkoy_a}.

%  \item \textbf{Conservation laws}:
The Euler system of equations  given by Eq.~\eqref{Eulern} \emph{preserves mass and momentum.} In particular, the system \eqref{Eulern} can be rewritten as
\begin{equation}\label{Eulern:conserve}
    \begin{aligned}
      \rho_t + \nabla \cdot(\rho \bm{u}) &= 0,\\
      (\rho\bm{u})_t + \nabla \cdot(\rho \bm{u}\otimes \bm{u}) &= \rho \mathcal{L}(\rho \bm{u}) - \rho \bm{u} \mathcal{L}(\rho) .
    \end{aligned}
\end{equation}
Obviously, the \emph{conservation of mass} can be obtained by integrating the first equation of the above system over $\mathbb{R}^n$. In addition, because the nonlocal operator $\mathcal{L}$ is assumed self-adjoint, i.e.,
 $$\int_{\mathbb{R}^n}\left(\rho \mathcal{L}(\rho \bm{u}) - \rho \bm{u} \mathcal{L}(\rho)  \right) = 0,$$
we have  the \emph{conservation of momentum} by integrating the second equation of the above system over $\mathbb{R}^n$.
\end{enumerate}

\section{Numerical methods}\label{sec:2}
\subsection{Microscale: Agent-based model}
In this subsection, we briefly introduce how to use the particle method to solve the agent-based model in the form of Eq.~\eqref{eq:Micro}.
%\subsection{One-dimensional case}
Let us first consider the one-dimensional case; to do this, we use the following particle method:
Assume that the macroscale density and velocity of a large number of particles are $\rho_0(x)$ and $v_0(x)$ satisfying
$$\int_{\mathbb{R} } \rho_0(x)dx = 1,$$
and the support of the density $\rho(x) $ is $\Lambda: = [a,b]$. We then obtain the initial position and velocity of each individual particle as follows:
\begin{itemize}
  \item Step 1. We first sample particles in the domain according to the initial density, i.e., give the initial position of each particle. To this end, we divide the domain $\Lambda $ into $P$ non-overlapping subdomains $\Lambda_p : = [x_{p-1},x_{p}]$ with $a = x_0< x_1<\ldots <x_{P} = b$ and compute the density in each subdomain $\Lambda_p$:
  $$\rho_p = \int_{x_{p-1}}^{x_{p}} \rho_0(x), \quad p = 1,2,\ldots, P.$$
  \item Step 2. Then we have that the number of particles sampled in each subdomain is $N_p = N\cdot \rho_p, p =1,2,\ldots,P$, where $N\gg 1$ is the total number of sampled particles. For each subdomain $\Lambda_p$, we divide it into $N_p$ uniform subintervals $[x_{p,q}, x_{p,q+1}]$ with  $x_{p-1} = x_{p,0}< x_{p,1}<\ldots <x_{p,N_p} = x_{p}$ and take the middle point of each subinterval, i.e., $(x_{p,q}+ x_{p,q+1})/2$, as the position of the particle.
  \item Step 3. We now compute the initial velocity of each particle. Since we have already assigned an initial position for each particle, then its initial velocity can be computed directly, i.e., for each particle $x_i,\,i= 1,\ldots, N$, its initial velocity is given by $v_i = v_0(x_i)$.
\end{itemize}

Given initial positions and velocities for all particles, we then use the equation~\eqref{eq:Micro} for the time integration. To this end, we use the \emph{Velocity Verlet} (VV) method~\cite{2018Li}, which is a most commonly used method due to the symplecticity, numerical stability and ease of implementation. It integrates position using half-step values of velocity. In particular, we let
$$a(x,v) = F(x,v) = \frac{dv(t)}{dt}. $$
Since we have set the mass equals to 1, at time $t_n = n\Delta t$ with $\Delta t = t_{n+1} -t_{n}, \, n = 0, 1, \ldots$, the VV scheme has the following form
\begin{equation}\label{VV}
\begin{aligned}
 &v^{n+1/2} = v^n  + a^n \Delta t /2, \; &x^{n+1} = x^{n} + v^{n+1/2} \Delta t, \; \\
 &a^{n+1} = F(x^{n+1}, v^{n+1/2}) ,\; &v^{n+1} = v^{n} + a^{n+1}\Delta t/2,
\end{aligned}
\end{equation}
where $x^n, \, v^n$ are the numerical solutions of the position and velocity at time $t_n$.

For the two-dimensional case, the procedure is the same as the one for the one-dimensional case. We thus omit the details here.

\subsection{Macroscale: Euler system of equations }
In the last subsection, we solve the agent-based model using the particle method, while in this subsection, we use the \emph{finite volume method} (FVM) to solve the macroscale model, i.e., the Euler equations~\eqref{Eulern}.
\subsubsection{One-dimensional case}
We firstly solve the Euler system of equations  \eqref{Eulern} using the FVM in one space dimension.
In particular, assuming the support of the density, i.e., the computational domain, is $\Lambda = [a,b]$, we consider the following one-dimensional system in conserved form which reads as
\begin{equation}\label{Euler1d:conserve1d}
    \bm{w}_t + \bm{f}(\bm{w})_x = \bm{s}(\bm{w}), \quad t\ge0, \;x\in \Lambda,
\end{equation}
where
$$\bm{w} = [\rho, m]^T,\; \bm{f}(\bm{w}) = [m, \rho u^2]^T,\; \bm{s}(\bm{w}) = [0,\rho \mathcal{L}(m) - m \mathcal{L}(\rho)]^T $$
with $m = \rho u$,
and discretize it by using the finite volume discretization.

%\subsubsection{Finite volume discretization}
Now we introduce the FVM for the above equation.
Firstly, let us divide the computational domain $\Lambda$ into $K$ non-overlapping subintervals
\begin{equation*}
    a = x_{1/2}<x_{3/2}< \cdots < x_{K + 1/2} = b
\end{equation*}
and denote
$$I_j = (x_{j-1/2}, x_{j+1/2}),\, j = 1,\ldots, K. $$
Let $\Delta x = \Delta x_j = x_{j+1} -x_j$ and $\bar{\bm{w}}_j(t)$ be the cell average of $\bm{w}(x,t)$  in cell $I_j$, i.e.,
$$\bar{\bm{w}}_j(t) = \frac{1}{\Delta x}\int_{I_j} \bm{w}(x,t)dx,\, j = 1,\ldots, K.$$
A finite volume scheme based on the cell averages has the following form
\begin{equation}\label{DGform}
    \bar{\bm{w}}_j^{n+1} = \bar{\bm{w}}_j^{n} - \lambda \left[ \bm{h}(\bm{w}_{j+\frac{1}{2}}^{-,n},\bm{w}_{j+\frac{1}{2}}^{+,n}) - \bm{h}(\bm{w}_{j-\frac{1}{2}}^{-,n},\bm{w}_{j-\frac{1}{2}}^{+,n}) \right] + \lambda\int_{I_j} \bm{s}(\bm{w}(x))dx,
\end{equation}
where $\lambda = \frac{\Delta t}{\Delta x}$ with $\Delta t$ is the time step, $\bar{\bm{w}}_j^{n}$ is the numerical solution of the cell average of the cell $I_j$ at time $t_n = n\Delta t$ and $\bm{h}(\cdot,\cdot)$ is a numerical flux. We assume  that at the cell interface $x= x_{j-1/2}$ we have two numerical approximations $\bm{w}_l=(\rho_l, m_l)^T$ and $\bm{w}_r=(\rho_r, m_r)^T$ from the left and right, respectively. We use in one-dimension the Godunov flux given by~\cite{2013Yang, Bouchut2003SINUM}
\begin{equation*}
    (\bm{h}(\bm{w}_l,\bm{w}_r))^T = (\widehat{\rho u}_{j-1/2}, \widehat{\rho u^2}_{j-1/2})
    = \left\{
\begin{array}{llllll}
(m_l,\rho_l u_l^2),  & u_l>0,\, u_r>0, \\
(0,0), \quad & u_l\le 0,\, u_r>0, \\
(m_r,\rho_r u_r^2),  & u_l \le 0,\, u_r \le 0,\\
(m_l,\rho_l u_l^2),  & u_l > 0,\, u_r \le 0,\, v>0, \\
(m_r,\rho_r u_r^2),  & u_l > 0,\, u_r \le 0,\, v<0, \\
(\frac{m_l+m_r}{2}, \frac{\rho_l u_l^2 + \rho_r u_r^2}{2}),  & u_l > 0,\, u_r \le 0,\, v=0,
\end{array}
\right.
\end{equation*}
where
\begin{equation*}
    u_l = \frac{m_l}{\rho_l}, \quad u_r = \frac{m_r}{\rho_r}, \quad v = \frac{\sqrt{\rho_l}u_l + \sqrt{\rho_r}u_r}{\sqrt{\rho_l} + \sqrt{\rho_r}}.
\end{equation*}

%\subsubsection{Computation of the nonlocal term}
We now discuss the discretizaiton of the nonlocal term $\bm{s}(\bm{w}(x))$.
To do this, we need to discretize the fractional Laplacian of a given function.
For a given function $f(x)$ defined on $\mathbb{R}$ with compact support $\Lambda$, we write the nonlocal operator $\mathcal{L}(f)$ in the following equivalent form  %~\cite[Lemma 3.2]{DPV2012}:
\begin{equation*}\label{FL2}
    \mathcal{L}(f)  = \int_{\mathbb{R}^n} \left(f(x+z) - f(z)\right){\phi(|z|)} dz.
\end{equation*}
This formula makes it easy to discretize $\mathcal{L}(f)$. We approximate the above integral by using the rectangle quadrature rule. In order to compute $\mathcal{L}(f)$, we will use all the information in the whole domain, i.e., all the values of $f_k,\,k\in \mathbb{Z}$.
In practical applications, the considered function $f$ (in our work, it is density $\rho$ or momentum $m$) is usually constant  or vanishes outside a finite domain.
We therefore discretize $\mathcal{L}(f)$ using the cell $I_k,\, k\in \mathbb{Z}$ up to $|y| = \widehat{K} \Delta x$, which goes to $\infty$ as $\Delta x \rightarrow 0$ and approximate the remaining parts with two unbounded space steps $(-\infty, -\widehat{K} \Delta x)$, and $(\widehat{K} \Delta x,\infty)$. For instance, we define
$$g(\rho,m)_j :=  \int_{I_j} \rho \mathcal{L}(m) dx $$
and approximate it as follows:
%where $a = \rho, \;b = m$ or $a = m ,\; b = \rho$.
Since $\rho$ is a constant in each cell, then we have for $j = 1,\ldots, K$,
\begin{equation*}
\begin{aligned}
    & g(\rho,m)_j =  \bar{\rho}_j \int_{I_j} dx \int_{\mathbb{R}} \left({m(x+z)  - m(x)}\right){\phi(|z|)} dz\\
                 = &\bar{\rho}_j \int_{I_j} dx \int_{0<|z|\le \widehat{K}\Delta x} \left({m(x+z)  - m(x)}\right){\phi(|z|)} dz \\
                  &+ \bar{\rho}_j \int_{I_j} dx \int_{|z|> \widehat{K}\Delta x} \left({m(x+z)  - m(x)}\right){\phi(|z|)} dz\\
                \approx & \bar{\rho}_j \Delta x^2 \sum_{0<k\le \widehat{K}} \left({\bar{m}_{j+k} +\bar{m}_{j-k} - 2\bar{m}_{j}}\right) {\phi_k }
                + \bar{\rho}_j \Delta x \left({\bar m_{j+\widehat{K}+1} + \bar m_{j-\widehat{K}-1} - 2\bar m_{j}} \right){\phi_{\widehat{K}+1}},
\end{aligned}
\end{equation*}
where $\phi_k = \phi(|k\Delta x|)$ if $k\le \widehat{K}$ and  $\phi_{\widehat{K}+1} = \int_{z> \widehat{K}\Delta x} \phi(|z|) dz$, which can be calculated directly or by using Gauss quadrature (e.g., Gauss Laguerre quadrature).
The above equation can be written in a matrix form:
\begin{equation}\label{nonlocal:m}
    g(\rho,m)_j \approx  \bar{\rho}_j \Delta x \left(G \vec{m} \right)_j,
\end{equation}
where $\vec{m} = \{\bar{m}_{j}\}_{1\le j\le K}$, and
$\left(G \vec{m} \right)_j$ denotes the $j$-th component of $G \vec{m}$,
$G$ is a matrix whose entries are given by: for $j = 1,\ldots, K$ and $i = \pm 1, \pm 2, \ldots,$
\begin{equation*}
    %\begin{aligned}
      G_{jj} = -2\sum_{k=1}^{\widehat{K}} \Delta x\phi_{k} -2 \phi_{\widehat{K}+1},\;
      G_{j,j+i} = \Delta x\phi_{|i|}.
     % G_{j,j+K+1} = G_{j,j-K-1} =  \phi_{K+1}.
    %\end{aligned}
\end{equation*}
Since the kernel function $\phi(|z|)$ is a positive and symmetric function, we can observe  from the above that $G$  has the following properties:
\begin{enumerate}
  \item $G$ is symmetric;
%  \item $G$ is diagonal dominate, more precisely, it holds
%  $$\sum_{|k|=0}^{K+1} G_{j,j+|k|} = 0; $$
  \item $G$ is a Toeplitz matrix.
\end{enumerate}
%The third property of $G$ suggests that we can use a fast matrix-vector product to calculate $G\vec{m}$ based on FFT.
Similarly, we have
\begin{equation}\label{nonlocal:rho}
    g(m,\rho)_j \approx \bar{m}_j \Delta x \left(G \vec{\rho} \right)_j, \; j = 1,\ldots, K.
\end{equation}
Thus, combining equations \eqref{DGform}, \eqref{nonlocal:m} and \eqref{nonlocal:rho}, we have the following finite volume scheme: for $j = 1,\ldots,K$,
\begin{equation}\label{DGform1}
    \bar{\bm{w}}_j^{n+1} = \bar{\bm{w}}_j^{n} - \lambda \left[ \bm{h}(\bm{w}_{j+\frac{1}{2}}^{-,n},\bm{w}_{j+\frac{1}{2}}^{+,n}) - \bm{h}(\bm{w}_{j-\frac{1}{2}}^{-,n},\bm{w}_{j-\frac{1}{2}}^{+,n}) \right] + \bm{g}(\rho,m)_j,
\end{equation}
where $\bm{g}(\rho,m)_j = \left[0,\lambda \Delta x (\bar{\rho}_j (G \vec{m})_j - \bar{m}_j(G \vec{\rho} )_j )\right]^T$.
%However, I am not sure there would be similar structure for $G$ in high order scheme.

%\textbf{\emph{A detail description of the fast matrix vector production would be added (Put it in Appendix?)}}

Note that the matrix $G$ is a Toeplitz matrix, therefore, this allows us to use the fast matrix vector procedure to compute $G \vec{m}$ or $G \vec{\rho}$ in $O(K\log K)$ operation~\cite{2014Wang}. Furthermore,
since the matrix $G$ is symmetric, we can show that the above scheme possesses mass and momentum conservation laws. We state it in the following theorem:
\begin{thm}
Assuming  that $\bm{w}$ has a compact support or periodic boundary conditions in the domain $\Lambda$, then the mass and momentum of the numerical solutions of~\eqref{DGform1} are conserved, i.e.,  for $n>1$, it holds
\begin{eqnarray}
  \int_{\Omega} \rho^n(x)dx &=& \int_{\Omega} \rho_0(x) dx; \label{conservation:mass} \\
  \int_{\Omega} \rho^n(x)u^n(x)dx &=& \int_{\Omega} \rho_0(x)u_0(x) dx. \label{conservation:momentum}
\end{eqnarray}
\end{thm}

\begin{proof}
Denote $\bm{h}_{j+1/2}^n: = \bm{h}(\bm{w}_{j+\frac{1}{2}}^{-,n},\bm{w}_{j+\frac{1}{2}}^{+,n})$, we can rewritten \eqref{DGform1} as:
$$\bar{\bm{w}}_j^{n+1} - \bar{\bm{w}}_j^{n} + \lambda \left[ \bm{h}_{j+1/2}^n - \bm{h}_{j-1/2}^n \right]  =  \bm{g}(\rho,m)_j.$$
Summing  over $j$ and noting  that $\bm{w}$ has a compact support or periodic boundary conditions, we obtain
\begin{equation}\label{conservepf1}
    \int_{\Omega} \bm{w}^{n+1}(x)dx - \int_{\Omega} \bm{w}^n(x)dx   = \bm{g}(\rho,m)_j.
\end{equation}
Then the conservation of mass, i.e., equation \eqref{conservation:mass} can be readily obtained since the first component of $\bm{g}(\rho,m)_j$ is equal to zero. For the conservation of momentum, i.e., equation \eqref{conservation:momentum}, we proceed as follows: We have that the second component of $\bm{g}(\rho,m)_j$ is given by
$$\lambda \Delta x \sum_{j} \left(  \bar{\rho}_j \left(G \vec{m} \right)_j - \bar{m}_j \left(G \vec{\rho} \right)_j \right).$$
Note that $G$ is symmetric, then we have
\begin{equation*}
\begin{aligned}
  \sum_{j} \bar{\rho}_j \left(G \vec{m} \right)_j
  = &\sum_{j} \sum_{k}\bar{\rho}_j G_{jk}\bar{m}_k
  = \sum_{k} \sum_{j} \bar{m}_k G_{jk}\bar{\rho}_j \\
   = &\sum_{k} \sum_{j} \bar{m}_k G_{kj}\bar{\rho}_j
   = \sum_{k}\bar{m}_k \left(G \vec{\rho} \right)_k.
\end{aligned}
\end{equation*}
Thus, the second component of $\bm{g}(\rho,m)_j$ vanishes and the result follows from \eqref{conservepf1}.
\end{proof}

\subsubsection{Two-dimensional case}
Now let us consider the Euler system of equations  \eqref{Eulern} in two space dimensions.
Assuming the support of the density, i.e., the computational domain, is a rectangle domain $\Omega = [a,b]\times [c,d]$,
then the two-dimensional Euler equations in conservative form  read as
\begin{equation}\label{Euler1d:conserve1d}
    \bm{w}_t + \bm{f}(\bm{w})_x  + \bm{g}(\bm{w})_y= \bm{s}(\bm{w}), \quad t\ge0,\;  x\in \Omega,
\end{equation}
where
\begin{equation*}
\begin{aligned}
  &\bm{w} = [\rho, m_1, m_2]^T,\; \bm{f}(\bm{w}) = [m_1, \rho u^2, \rho uv]^T,\; \bm{g}(\bm{w}) = [n, \rho uv, \rho v^2]^T,\\
  &\bm{s}(\bm{w}) = [0,\rho \mathcal{L}(m_1) - m_1 \mathcal{L}(\rho), \rho \mathcal{L}(m_2) - m_2 \mathcal{L}(\rho)]^T
\end{aligned}
\end{equation*}
with $m_1 = \rho u$ and $m_2 = \rho v$.

%\subsubsection{Finite volume discretization}
We now solve the above system by using the finite volume scheme.
Similarly as we did for the one-dimensional case, we use the uniform rectangular meshes, namely, $\Delta x = \Delta x_{i} = x_{i+1/2} -x_{i-1/2}, \; \Delta y = \Delta y_{j} = y_{j+1/2} -y_{j-1/2}$, $i = 1,\ldots,K,\, j = 1,\ldots, L$,
where
\begin{equation*}
    a = x_{1/2}<x_{3/2}<\ldots< x_{K+1/2} = b, \; c = y_{1/2}<y_{3/2}<\ldots< y_{L+1/2} = d.
\end{equation*}
Let the cell $I_{ij}$ be defined by $I_{ij} := [{x_{i-1/2}},{x_{i + 1/2}}] \times[{y_{j-1/2}},{y_{j + 1/2}}]$, $\Delta t$ be the time step, and  $\bar{\bm{w}}_{ij}(t)$ be the cell average of $\bm{w}(x,y,t)$ in the cell $I_{ij} $, i.e.,
$$\bar{\bm{w}}_{ij}(t) = \frac{1}{\Delta x \Delta y}\int_{x_{i-1/2}}^{x_{i+1/2}} \int_{y_{j-1/2}}^{y_{j+1/2}} \bm{w}(x,y,t)dxdy,\, i = 1,\ldots,K,\, j = 1,\ldots, L.$$
A finite volume scheme based on the cell averages has the following form
\begin{equation}\label{DGform2d}
\begin{aligned}
    \bar{\bm{w}}_{ij}^{n+1} = &\bar{\bm{w}}_{ij}^{n} - \lambda \int_{y_{j-1/2}}^{y_{j+1/2}} \left[ \bm{h}_1(\bm{w}_{i+\frac{1}{2},j}^{-,n},\bm{w}_{i+\frac{1}{2},j}^{+,n}) - \bm{h}_1(\bm{w}_{i-\frac{1}{2},j}^{-,n},\bm{w}_{i-\frac{1}{2},j}^{+,n}) \right]dy\\
    & - \lambda \int_{x_{i-1/2}}^{x_{i+1/2}} \left[ \bm{h}_2(\bm{w}_{i,j+\frac{1}{2}}^{-,n},\bm{w}_{i,j+\frac{1}{2}}^{+,n}) - \bm{h}_2(\bm{w}_{i,j-\frac{1}{2}}^{-,n},\bm{w}_{i,j-\frac{1}{2}}^{+,n}) \right]dx\\
     &+ \lambda \int_{x_{i-1/2}}^{x_{i+1/2}} \int_{y_{j-1/2}}^{y_{j+1/2}} \bm{s}(\bm{w}(x,y,t_n))dxdy,
\end{aligned}
\end{equation}
where $\lambda = \frac{\Delta t}{\Delta x \Delta y}$,  $\bm{h}_1(\cdot,\cdot), \,\bm{h}_2(\cdot,\cdot)$ are one-dimensional  numerical fluxes, and $\bar{\bm{w}}_{ij}^{n}$ is the numerical solution of the cell average of cell $I_{ij}$ at time $t_n$.
Suppose that $(x,y)= (x_{i-1/2},y_0)$ is a point on the vertical cell interface, at which we have two numerical approximations $\bm{w}_l=(\rho_l, (m_1)_l,(m_2)_l)^T$ and $\bm{w}_r=(\rho_r, (m_1)_r, (m_2)_r)^T$ from the left and right, respectively. For the two-dimensional case, we also use the Godunov flux given by~\cite{2013Yang}
\begin{equation*}
\begin{aligned}
   & (\bm{h}_1(\bm{w}_l,\bm{w}_r))^T = (\widehat{\rho u}_{j-1/2}, \widehat{\rho u^2}_{j-1/2}, \widehat{\rho u v}_{j-1/2})\\
    = &
    \left\{
\begin{array}{ll}
((m_1)_l,\rho_l u_l^2, \rho_l u_l v_l),  & u_l>0,\, u_r>0, \\
(0,0,0), \quad & u_l\le 0,\, u_r>0, \\
((m_1)_r,\rho_r u_r^2, \rho_r u_r v_r),  & u_l \le 0,\, u_r \le 0,\\
((m_1)_l,\rho_l u_l^2, \rho_l u_l v_l),  & u_l > 0,\, u_r \le 0,\, v>0, \\
((m_1)_r,\rho_r u_r^2, \rho_r u_r v_r),  & u_l > 0,\, u_r \le 0,\, v<0, \\
\frac{1}{2}((m_1)_l+(m_1)_r, \rho_l u_l^2 + \rho_r u_r^2, (m_1)_lv_l + (m_1)_rv_r),  & u_l > 0,\, u_r \le 0,\, v=0,
\end{array}
\right.
\end{aligned}
\end{equation*}
where
\begin{equation*}
    (u_l,v_l) = \left(\frac{(m_1)_l}{\rho_l}, \frac{(m_2)_l}{\rho_l}\right), \quad
    (u_r,v_r) = \left(\frac{(m_1)_r}{\rho_r}, \frac{(m_2)_r}{\rho_r}\right), \quad
    v = \frac{\sqrt{\rho_l}u_l + \sqrt{\rho_r}u_r}{\sqrt{\rho_l} + \sqrt{\rho_r}}.
\end{equation*}
The numerical flux $(\bm{h}_2(\bm{w}_l,\bm{w}_r))^T = (\widehat{\rho u}, \widehat{\rho uv}, \widehat{\rho v^2})$ can be defined similarly on the horizontal cell interfaces.

%\subsubsection{Computation of the nonlocal term}
%In this work, we only consider the piecewise constant approximation, i.e., $k = 0$.
Similarly to the one-dimensional problem, we now discuss how to discretize the nonlocal terms, i.e., the source terms $\bm{s}(\bm{w})$, for the two-dimensional problem.
In order to discretize $\bm{s}(\bm{w})$, we first need to discretize the nonlocal operator $\mathcal{L}(f)$ for a given function $f$. To this end, for a given function $f(\bm{x}),\bm{x}: = (x,y)\in \mathbb{R}^2$ compactly supported in the domain $\Omega$, we have that the nonlocal term $\mathcal{L}(f)$ is equivalent to the following formula:
\begin{equation*}\label{FL2d2}
    \mathcal{L}(f) =   \int_{|\bm z|>0} \left({f(\bm x+\bm z) - f(\bm z)}\right){\phi(|\bm z|)} d{\bm z},
\end{equation*}
where $\bm{z} = (z_1,z_2)$.
For $\nu = 1,2$, we define
$$g(\rho,m_\nu)_{ij} :=  \int_{I_{ij} } \rho \mathcal{L}(m_\nu) dx dy, \quad 1\le i\le K, \, 1\le j\le L. $$
Writing it into the vector form, we have the vector form for $g(\rho,m_\nu)_{ij}$
\begin{equation*}
\begin{aligned}
  \bm{g}(\rho,m_\nu) = & [g(\rho,m_\nu)_{11},\ldots, g(\rho,m_\nu)_{K1};g(\rho,m_\nu)_{12},\ldots, g(\rho,m_\nu)_{K2}; \\
  & \ldots; g(\rho,m_\nu)_{1L},\ldots, g(\rho,m_\nu)_{KL}]^T.
\end{aligned}
\end{equation*}
$\bm{g}(\rho,m_\nu)$ can be computed using similar arguments as the one used for the one-dimensional case. In particular,
\begin{equation*}
\begin{aligned}
    & g(\rho,m_{\nu})_{ij} =  \bar{\rho}_{ij} \int_{I_{ij}} dxdy \int_{\mathbb{R}} \left({m_{\nu}(x+z_1, y + z_2)  - m_{\nu}(x,y)}\right){\phi(|\bm{z}|)} dz_1dz_2\\
                 = &\bar{\rho}_{ij} \int_{I_{ij}} dxdy \int_{D_1}\left({m_{\nu}(x+z_1, y + z_2)  - m_{\nu}(x,y)}\right){\phi(|\bm{z}|)} dz_1dz_2 \\
                  &+ \bar{\rho}_{ij} \int_{I_{ij}} dxdy \int_{D_2} \left({m_{\nu}(x+z_1, y + z_2)  - m_{\nu}(x,y)}\right){\phi(|\bm{z}|)} dz_1dz_2\\
                \approx & \bar{\rho}_{ij} \Delta x^2 \Delta y^2
                \sum_{\substack{0<k\le \widehat{K}\\0<l\le \widehat{L}}} {\phi_{kl} }  \big((\bar{m}_{\nu})_{i+k,j+l} +(\bar{m}_{\nu})_{i-k,j+l} + (\bar{m}_{\nu})_{i+k,j-l} +(\bar{m}_{\nu})_{i-k,j-l} - 4(\bar{m}_{\nu})_{ij}\big)\\
                & + \bar{\rho}_{ij} \Delta x\Delta y  {\phi_{\widehat{K}+1,\widehat{L}+1}} \big( (\bar{m}_{\nu})_{i+\widehat{K}+1, j+\widehat{L}+1} + (\bar{m}_{\nu})_{i-\widehat{K}-1, j+\widehat{L}+1} \\
                & + (\bar{m}_{\nu})_{i+\widehat{K}+1, j-\widehat{L}-1} + (\bar{m}_{\nu})_{i-\widehat{K}-1, j-\widehat{L}-1} - 4(\bar{m}_{\nu})_{ij} \big),
\end{aligned}
\end{equation*}
where $D_1 = \{(z_1,z_2): 0<|z_1|\le \widehat{K}\Delta x,0<|z_2|\le \widehat{L}\Delta y\}$, $D_2 = \{(z_1,z_2): |z_1|> \widehat{K}\Delta x,|z_2|> \widehat{L}\Delta y\}$, and
\begin{align*}
  &\phi_{11} = \phi(\sqrt{(\Delta x/2)^2 + (\Delta y/2)^2}),\;
  \phi_{k1} = \phi(\sqrt{(k\Delta x)^2 + (\Delta y/2)^2}), \\
  &\phi_{1l} = \phi(\sqrt{(\Delta x/2)^2 + (l\Delta y)^2}),\;
  \phi_{kl} = \phi(\sqrt{(k\Delta x)^2 + (l\Delta y)^2}), 2\le k\le \widehat{K},\, 2\le l\le \widehat{L},\\
  &\phi_{\widehat{K}+1,\widehat{L}+1} = \int_{D_3} \phi(|\bm z|) d\bm z,
\end{align*}
where $D_3 = \{(z_1,z_2):z_1>0, z_2>0\} \setminus \{(z_1,z_2):0\le z_1\le \widehat{K} \Delta x, 0\le z_2\le \widehat{L}\Delta y\}$.
%$$\phi_{kl} = \phi(\sqrt{(k\Delta x)^2 + (l\Delta y)^2}), k\le \widehat{K},\, l\le \widehat{L}$$, $\phi_{\widehat{K}+1,\widehat{L}+1} = \int_{D_2} \phi(|\bm z|) d\bm z$ which can be calculated by using Gauss quadrature (e.x., Gauss Laguerre quadrature).
The above equation can be written in a matrix form:
\begin{equation}\label{nonlocal:m2d}
    \bm{g}(\rho,m_\nu) \approx \Delta x \Delta y \vec{\bm{\rho}} * \left(\bm{G} \vec{\bm{m}}_{\nu} \right),
\end{equation}
where the symbol $*$ means component-wise multiplication and
\begin{align*}
 &\vec{\bm{\rho}} = [\bar{\rho}_{11}, \ldots, \bar{\rho}_{K1};\bar{\rho}_{12},\ldots, \bar{\rho}_{K2}; \ldots; \bar{\rho}_{1L},\ldots, \bar{\rho}_{KL}]^T, \\
 &\vec{\bm{m}}_{\nu} = [(\bar{m}_{\nu})_{11}, \ldots, (\bar{m}_{\nu})_{K1};(\bar{m}_{\nu})_{12},\ldots, (\bar{m}_{\nu})_{K2}; \ldots; (\bar{m}_{\nu})_{1L},\ldots, (\bar{m}_{\nu})_{KL}]^T,
\end{align*}
%$\vec{\bm{m}}_{\nu} = \{(\bar{m}_{\nu})_{ij}\}_{1\le i \le K, 1\le j\le L}$,
and
%$\left(G \vec{m} \right)_j$ denotes the $j$-th component of $G \vec{m}$,
$\bm{G}$ is a block-Toeplitz-Toeplitz-block matrix given by
$$ \bm{G} =
\left(
  \begin{array}{cccc}
    G^{11} & G^{12} & \cdots & G^{1L} \\
    G^{21} & G^{22} & \cdots & G^{2L} \\
    \vdots & \vdots & \ddots & \vdots \\
    G^{L1} & G^{L2} & \cdots & G^{LL} \\
  \end{array}
\right),
$$
where $G^{ij}, 1\le i,j\le L$ are Toeplitz matrices, where for $i = 1,\ldots, L\, ,j = \pm 1,\pm 2,\ldots$ and $m = 1,\ldots,K,\, n = \pm 1,\pm 2,\ldots,$
\begin{align*}
  &G^{ii}_{mm} = -4\sum_{k=1}^{\widehat{K}} \sum_{l=1}^{\widehat{L}} \Delta x\Delta y\phi_{kl} -4 \phi_{\widehat{K}+1,\widehat{L}+1},\;
  G^{ii}_{m,m+n} = \Delta x\Delta y\phi_{|n|1}, \\
  &G^{i,i+j}_{m,m+n} = \Delta x\Delta y \phi_{|m| |j|}.
\end{align*}
Similarly, we have
\begin{equation}\label{nonlocal:rho2d}
    \bm{g}(m_\nu,\rho) \approx \Delta x \Delta y \vec{\bm{m}}_{\nu} * \left(\bm{G} \vec{\bm{\rho}} \right).
\end{equation}
Thus, by combining equations \eqref{DGform2d}, \eqref{nonlocal:m2d} and \eqref{nonlocal:rho2d}, we have the following finite volume scheme for the two-dimensional problem:
\begin{equation}\label{DGform2d2}
\begin{aligned}
    \bar{\bm{w}}_{ij}^{n+1} = &\bar{\bm{w}}_{ij}^{n} - \lambda \int_{y_{j-1/2}}^{y_{j+1/2}} \left[ \bm{h}_1(\bm{w}_{i+\frac{1}{2},j}^{-,n},\bm{w}_{i+\frac{1}{2},j}^{+,n}) - \bm{h}_1(\bm{w}_{i-\frac{1}{2},j}^{-,n},\bm{w}_{i-\frac{1}{2},j}^{+,n}) \right]dy\\
    & - \lambda \int_{x_{i-1/2}}^{x_{i+1/2}} \left[ \bm{h}_2(\bm{w}_{i,j+\frac{1}{2}}^{-,n},\bm{w}_{i,j+\frac{1}{2}}^{+,n}) - \bm{h}_2(\bm{w}_{i,j-\frac{1}{2}}^{-,n},\bm{w}_{i,j-\frac{1}{2}}^{+,n}) \right]dx\\
     &+ \lambda \bm g(\rho, m_1 , m_2)_{ij},
\end{aligned}
\end{equation}
where $$\bm g(\rho, m_1 , m_2)_{ij} = \left[0, [\bm{g}(\rho,m_1)]_{ij} - [\bm{g}(m_1,\rho)]_{ij}, [\bm{g}(\rho,m_2)]_{ij} - [\bm{g}(m_2,\rho)]_{ij} \right]^T,$$
here $[V]_{ij} = [reshape(V,K,L)]_{ij}$ for a given $K\times L$ column vector $V$, where `\emph{reshape}' is a Matlab function.

Since $\bm{G}$ is a block-Toeplitz-Toeplitz-block matrix, so the computation of $\bm g(\rho, m_1 , m_2)$ can be implemented in $O(KL\log(KL))$ operations by using the fast matrix vector multiplication~\cite{WangTian2014CMAME}.

%\textbf{\emph{A detail description of the fast matrix vector production would be added (Put it in Appendix?)}}

\begin{thm}
Assuming  that $\bm{w}$ has a compact support or periodic boundary conditions in the domain $\Omega$, then the mass and momentum of the numerical solution of \eqref{DGform2d2} are conserved, i.e.,  for $n>1$, it holds
\begin{eqnarray}
  \int_{\Omega} \rho^n(x,y)dx dy &=& \int_{\Omega} \rho_0(x,y) dx dy; \label{conservation:mass2d} \\
  \int_{\Omega} \rho^n(x,y)u^n(x,y)dx dy &=& \int_{\Omega} \rho_0(x,y)u_0(x,y) dx dy; \label{conservation:momentumu}\\
  \int_{\Omega} \rho^n(x,y)v^n(x,y)dx dy &=& \int_{\Omega} \rho_0(x,y)v_0(x,y) dx dy. \label{conservation:momentumv}
\end{eqnarray}
\end{thm}

\begin{proof}
Denote $\bm{h}_{1,i+1/2}^n: = \bm{h}_1(\bm{w}_{i+\frac{1}{2},j}^{-,n},\bm{w}_{i+\frac{1}{2},j}^{+,n})$ and $\bm{h}_{2,j+1/2}^n: = \bm{h}_2(\bm{w}_{i,j+\frac{1}{2}}^{-,n},\bm{w}_{i,j+\frac{1}{2}}^{+,n})$, we can rewrite \eqref{DGform2d2} as:
\begin{equation*}
\begin{aligned}
    & \bar{\bm{w}}_{ij}^{n+1} - \bar{\bm{w}}_{ij}^{n}
    + \lambda \int_{y_{j-1/2}}^{y_{j+1/2}} \left[ \bm{h}_{1,i+1/2}^n - \bm{h}_{1,i-1/2}^n \right]dy + \lambda \int_{x_{i-1/2}}^{x_{i+1/2}} \left[\bm{h}_{2,j+1/2}^n - \bm{h}_{2,j-1/2}^n \right]dx\\
    =  & \lambda \bm g(\rho, m_1 , m_2)_{ij}.
\end{aligned}
\end{equation*}
Summing over $i$ and $j$, and noting that $\bm{w}$ has a compact support or periodic boundary conditions in the domain $\Omega$, we obtain
\begin{equation}\label{2dconservepf1}
    \int_{\Omega} \bm{w}^{n+1}(x)dx - \int_{\Omega} \bm{w}^n(x)dx   = \lambda \bm g(\rho, m_1 , m_2)_{ij}.
\end{equation}
Since $G$ is symmetric, by using the same argument as the one for the one-dimensional problem, we have that the scheme \eqref{DGform2d2} preserves both the mass and the momentum , i.e., equations \eqref{conservation:mass2d}-\eqref{conservation:momentumv} hold true.
\end{proof}

We mention here that for the time discretization, both in one- and two-dimension, we use a strong stability preserving scheme, namely, the second-order Runge-Kutta scheme to solve the resulted ODE system $\bm{w}_t = \bm{L}(\bm{w})$:
\begin{equation*}
%\begin{aligned}
   \bm{w}^{(1)} = \bm{w}^{n} + \Delta t \bm{L}(\bm{w}^{n}), \;
   \bm{w}^{n+1} = \frac{1}{2}\bm{w}^{n} + \frac{1}{2}\left(\bm{w}^{(1)} + \Delta t \bm{L}(\bm{w}^{(1)})\right).
%\end{aligned}
\end{equation*}

\section{Numerical examples}\label{sec:3}
In this section, we present numerical simulations, both in one- and two-dimensional nonlocal flocking dynamics, by solving the agent-based model~\eqref{eq:Micro} with particle method and the Euler system of equations ~\eqref{Eulern} with the FVM method. We compare the numerical solutions with the same initial conditions.
Our aim is to verify that the two models produce similar results and subsequently (in the next section)  to demonstrate how to infer the fractional order of the Laplacian using the trajectories obtained from the agent-based model.
We first consider the one-dimensional case.

\begin{figure}[b!]
\begin{center}
\begin{minipage}{0.46\linewidth}
\begin{center}
\includegraphics[scale=0.50,angle=0]{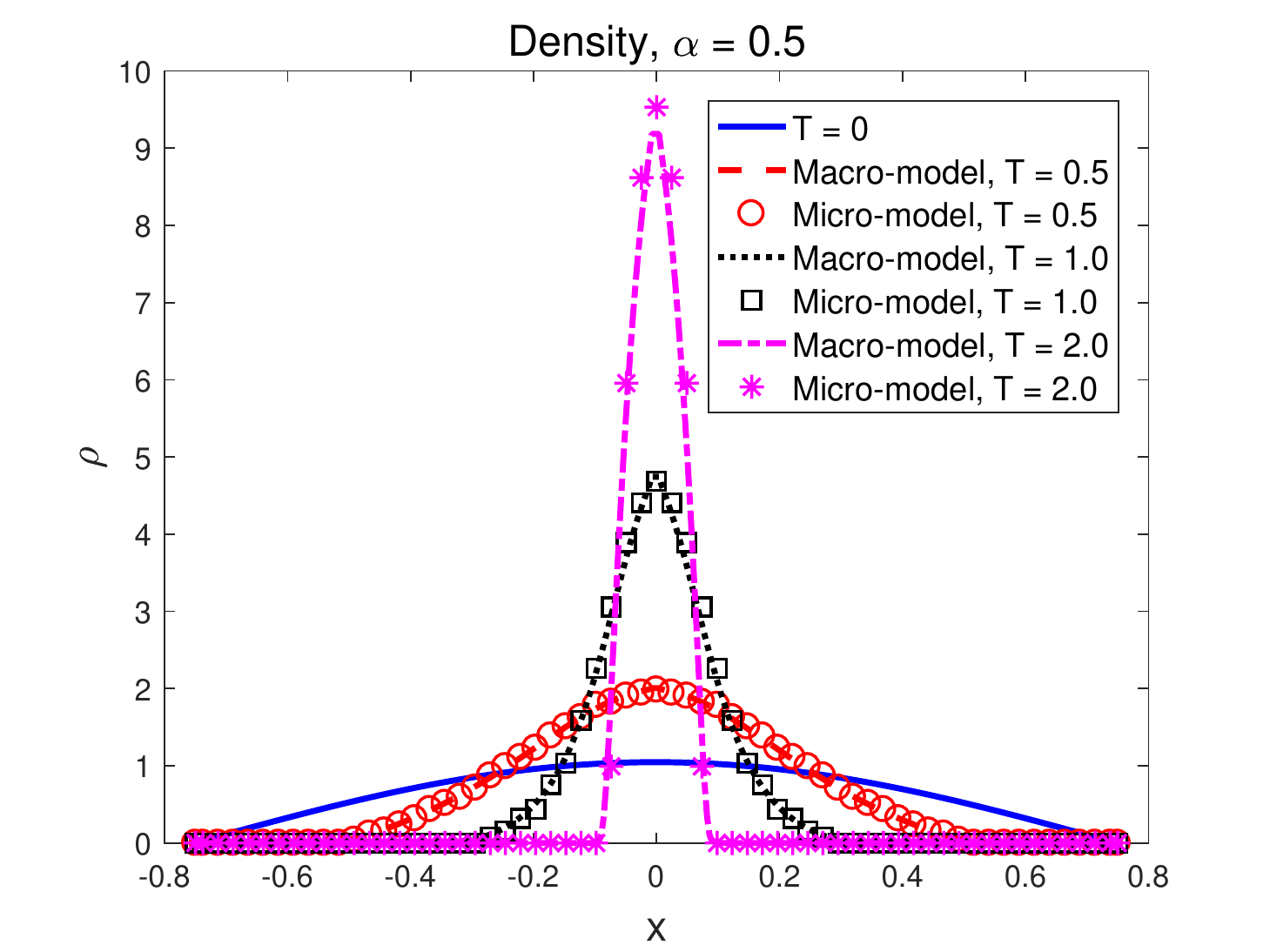}\end{center}
\end{minipage}
\begin{minipage}{0.46\linewidth}
\begin{center}
\includegraphics[scale=0.50,angle=0]{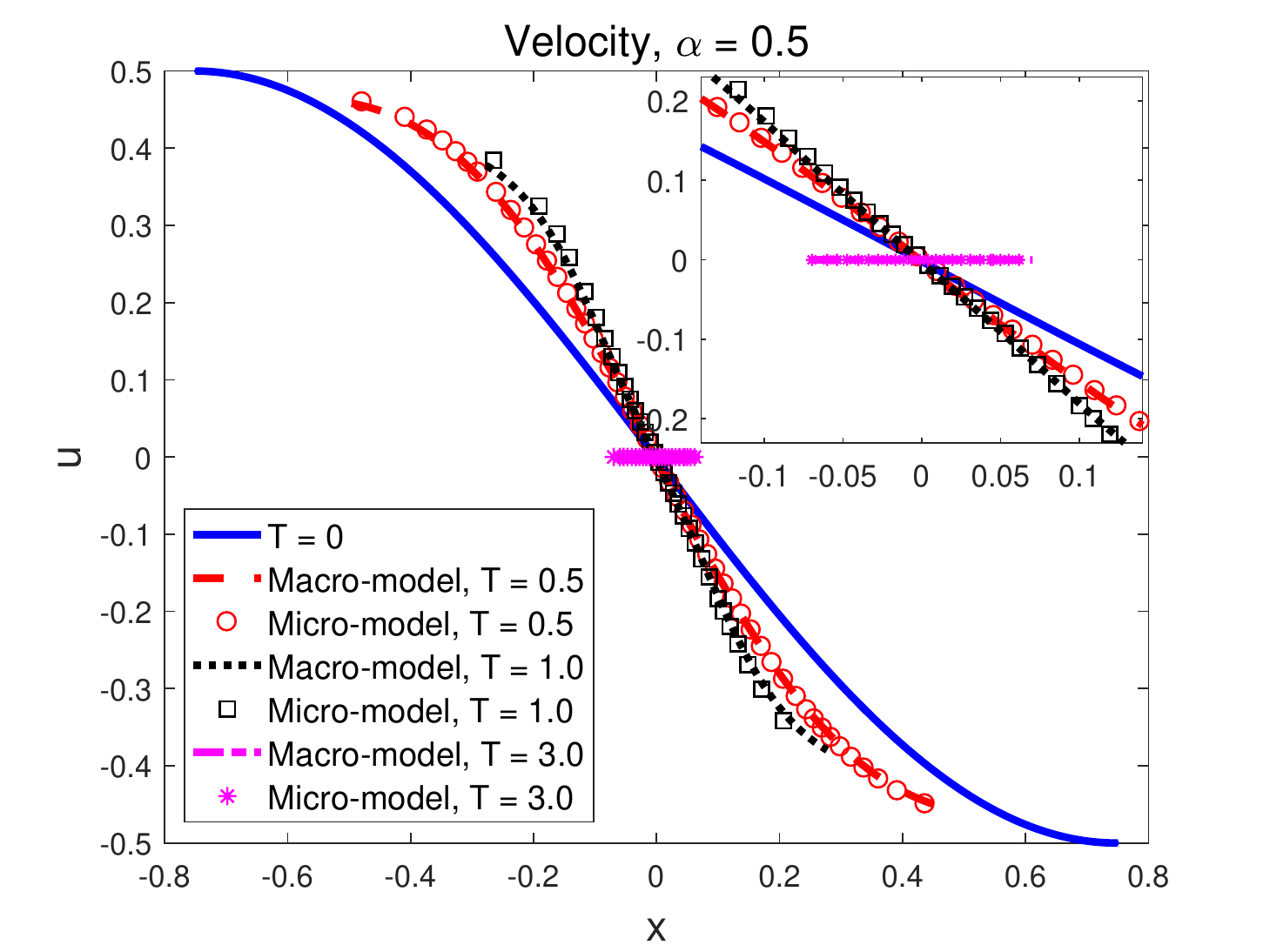}\end{center}
\end{minipage}
\caption{Validation of one-dimensional case for fractional order $\alpha = 0.5$: Numerical solutions of the density and velocity for the Microscale agent-based model \eqref{eq:Micro} (symbols) and the Macroscale Euler equations \eqref{Eulern} (lines) at different times $T$. Left: density, right: velocity. The inset in the right plot is for clarity in the comparison.
}\label{fig:Comparison1Da05}
\end{center}
\end{figure}

\begin{exam}\label{ex:1D:sln}
(\textbf{One-dimensional dynamics})
In this example, we consider the one-dimensional problem with the influence function given by
$$\phi(r) = \frac{c_{1,\alpha}}{r^{1+\alpha}}$$
and the initial density and velocity given by
\begin{equation}\label{Initial1}
    \rho_0(x) = \frac{\pi}{3} \cos\left(\frac{\pi x}{1.5} \right), \quad u_0(x) = -c \sin\left(\frac{\pi x}{1.5} \right),
\end{equation}
where $c = 0.5$ is a constant. Here we use the domain $\Lambda = [-0.75, 0.75]$. The boundary conditions for the density $\rho$ and mass $m = \rho u$ are homogenous Dirichlet boundary conditions.
\end{exam}

For the agent-based model, the total number of sampled particles for the simulation is $N = 1024$, while for the Euler equations  the space step is $\Delta x = 1/256$. The numerical solutions of the density and velocity at different times are shown in Figure \ref{fig:Comparison1Da05}-\ref{fig:Comparison1Da12} for the value of fractional order $\alpha = 0.5$ and $\alpha = 1.2$, respectively. We see that the velocity tends to a constant value, which is in agreement with the analytical result given in~\cite{2017Shvydkoy_a}.
%This means we have a steady state as $t\rightarrow \infty$.
Moreover, we observe from the comparison that the solutions of the microscale agent-based model coincide with solutions of the macroscale Euler equations. This means that the solution of the Euler system of equations can give a good prediction to the solution of the agent-based model.

\begin{figure}[b!]
\begin{center}
\begin{minipage}{0.46\linewidth}
\begin{center}
\includegraphics[scale=0.50,angle=0]{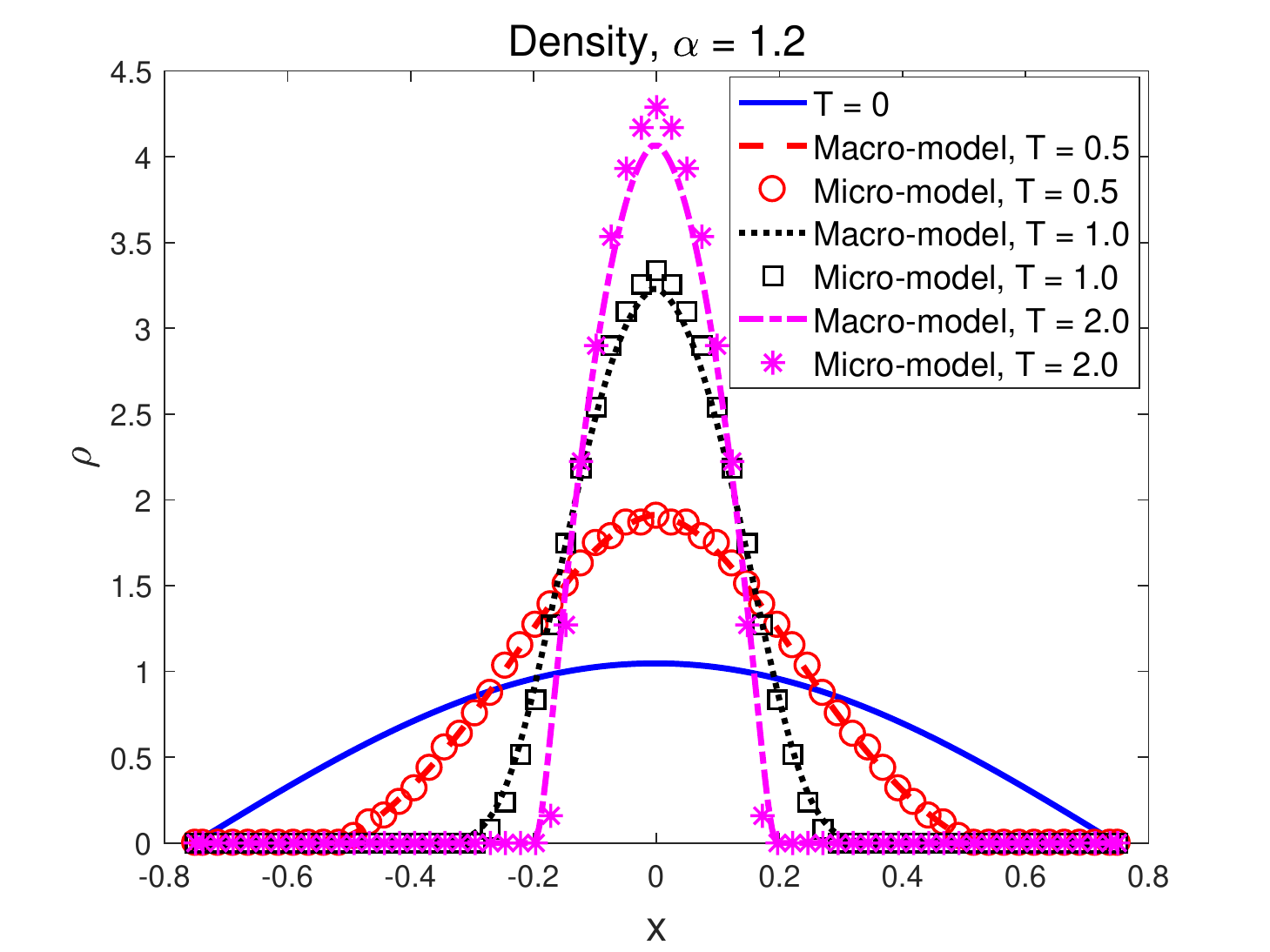}\end{center}
\end{minipage}
\begin{minipage}{0.46\linewidth}
\begin{center}
\includegraphics[scale=0.50,angle=0]{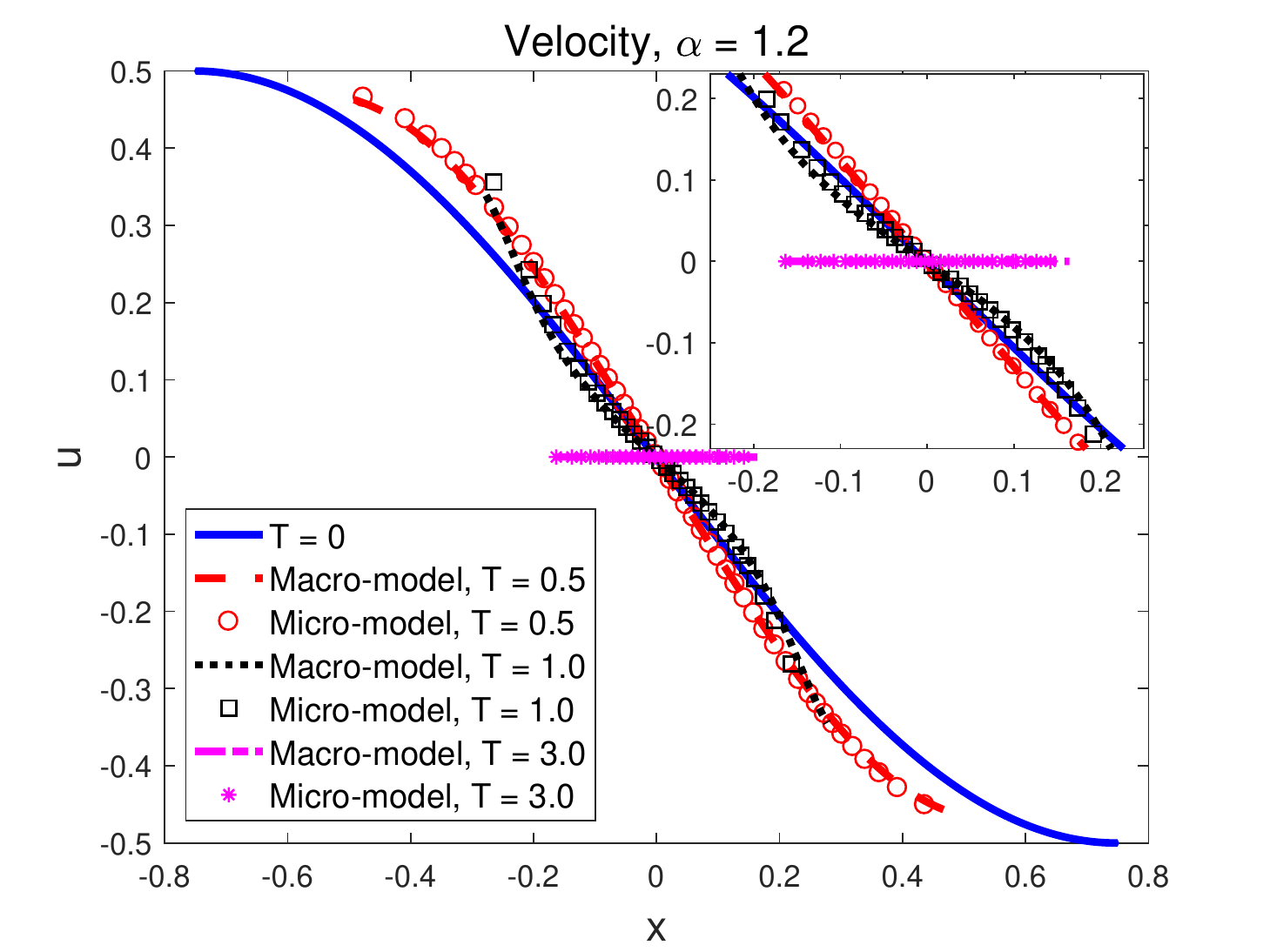}\end{center}
\end{minipage}
\caption{Validation of one-dimensional case for fractional order $\alpha = 1.2$: Numerical solutions of the density and velocity for the Microscale agent-based model \eqref{eq:Micro} (symbols) and the Macroscale Euler equations \eqref{Eulern} (lines) at different times $T$. Left: density, right: velocity. The inset in the right plot is for clarity in the comparison.
}\label{fig:Comparison1Da12}
\end{center}
\end{figure}

\begin{exam}\label{ex:2D:sln}
(\textbf{Two-dimensional dynamics})
We now consider the two-dimensional case, and consider the following influence function
$$\phi(r) = \frac{c_{2,\alpha}}{{r}^{2+\alpha}}.$$
The initial density and velocity are
\begin{equation}\label{Initial2}
\begin{aligned}
    &\rho_0(x,y) = \left(\frac{\pi}{3}\right)^2 \cos\left(\frac{\pi x}{1.5} \right) \cos\left(\frac{\pi y}{1.5} \right), \\
    &u_0(x,y) = -c \sin\left(\frac{\pi x}{1.5} \right), \; v_0(x,y) = -c \sin\left(\frac{\pi y}{1.5} \right),
\end{aligned}
\end{equation}
where $c = 0.5/\sqrt{2}$ is a constant. Here the domain is $\Omega = [-0.75, 0.75]^2$, and the boundary condition for $\rho$ and  $m_1 = \rho u, \, m_2 = \rho v$ are homogenous Dirichlet boundary conditions.
\end{exam}

In this case, the total number of sampled particles for the simulation is $N = 9976$ for the agent-based model, while for the Euler equations  the space step is $\Delta x = \Delta y = 1/64$. The numerical solutions of the density at different times are shown in Figure \ref{fig:Comparisondensity2Da05}-\ref{fig:Comparisondensity2Da12} for the value of fractional order $\alpha = 0.5$ and $\alpha = 1.2$, respectively. The relative differences of the density between the solution of the micro-model and the macro-model are also shown in the third row of each figure.
Again, we observe that solutions of the microscale agent-based model are in good agreement with solutions of the macroscale Euler system of equations.
This means that the numerical solution of the corresponding macro-model can yield correctly density distribution consistent with the collective behavior of the particle system.

\section{Infer the influence function using Gauss process machine learning}\label{sec:4}
The trajectories generated by the particle simulation mimic the field data of tracking logs that can be obtained experimentally, which is used to learn the value of the fractional order of the influence function using a Gaussian process regression model implemented with the Bayesian optimization.
As shown in the previous section, the solutions of the agent-based model~\eqref{eq:Micro} and the solutions of the Euler system of equations  \eqref{Eulern} are in agreement with each other. Thus, assuming that we have the solution of the agent-based model for a given influence function, we then infer the influence function by using the Euler system of equations  with a machine learning algorithm.
In this section, we show how to infer the influence function by using the \emph{Gaussian Process Machine Learning} (GPML) with the Bayesian optimization. This method has been used to discover the fractional order of fractional advection-dispersion equations in \cite{2017Pang}.
In particular, we are going to learn the influence function by solving the Euler equations  \eqref{Eulern}, while the data is obtained by solving the agent-based model \eqref{eq:Micro}. This is to say, we are going to use the agent-based model with a specified influence function to generate the data which is considered as the experimental data, then we use the numerical solution of the Euler equations  to infer the influence function by using the GPML approach.

\begin{figure}[t!]
    \centering
    \begin{subfigure}[b]{0.29\textwidth}
        \includegraphics[width=\textwidth]{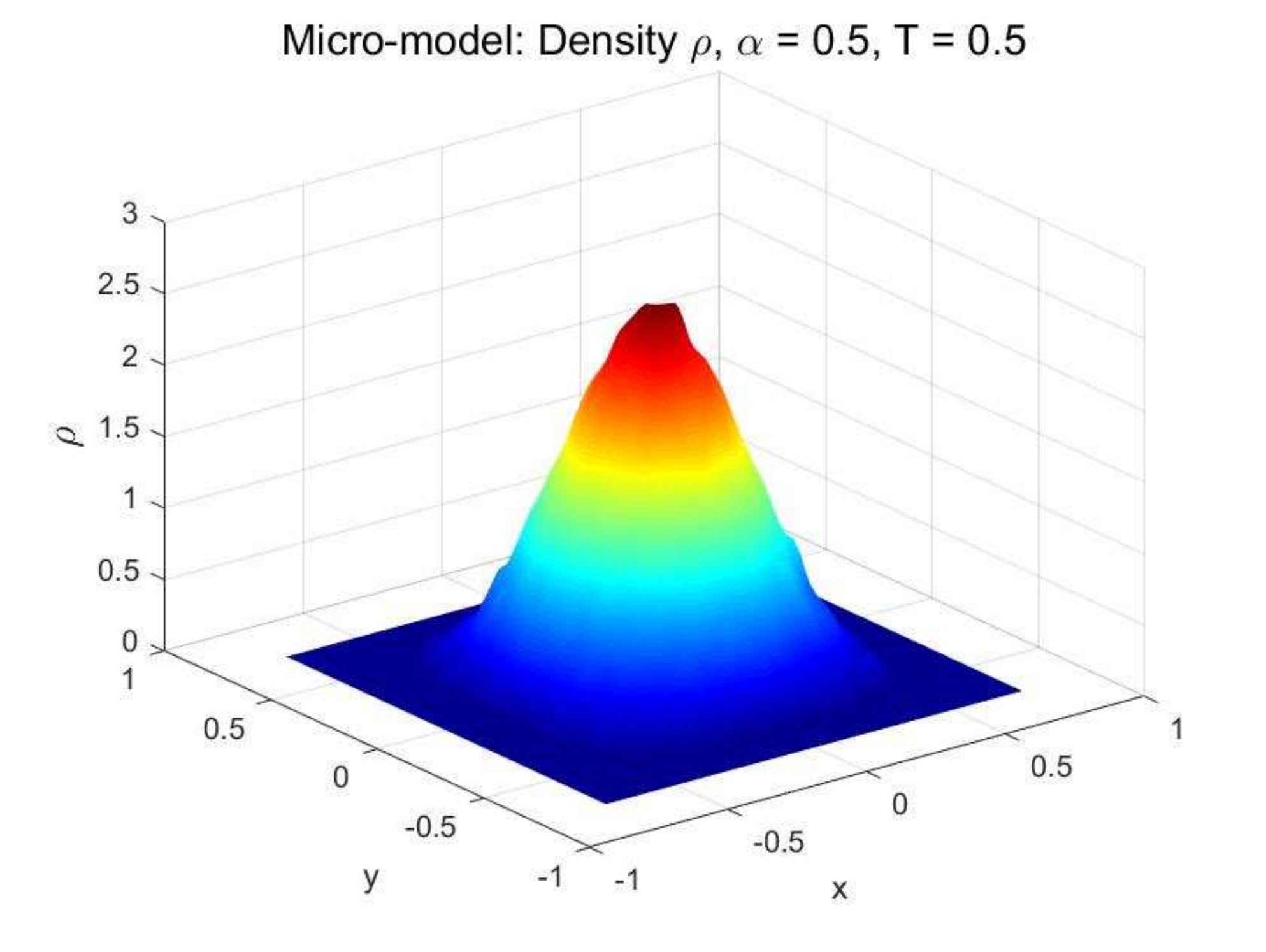}
        \includegraphics[width=\textwidth]{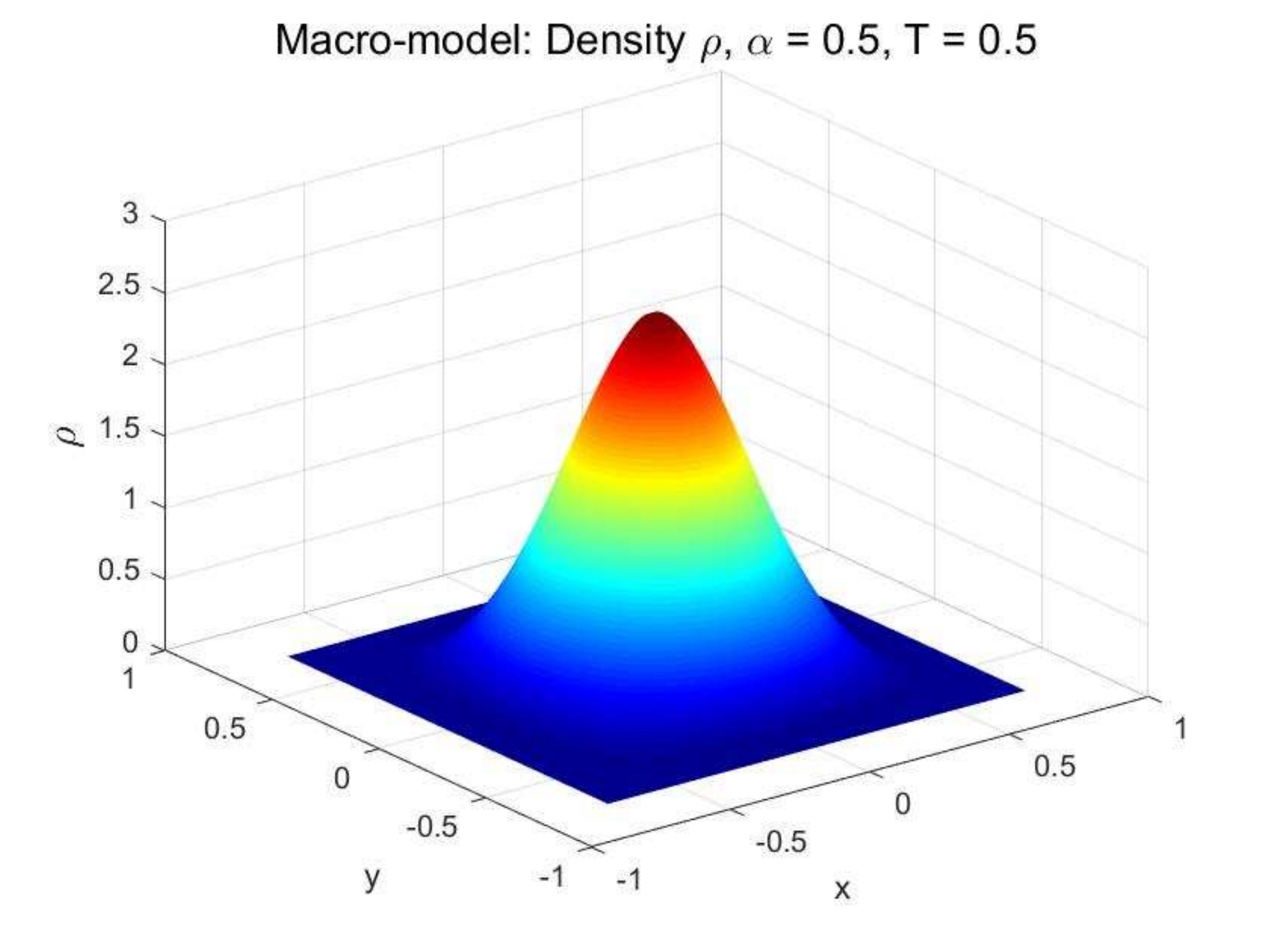}
        \includegraphics[width=\textwidth]{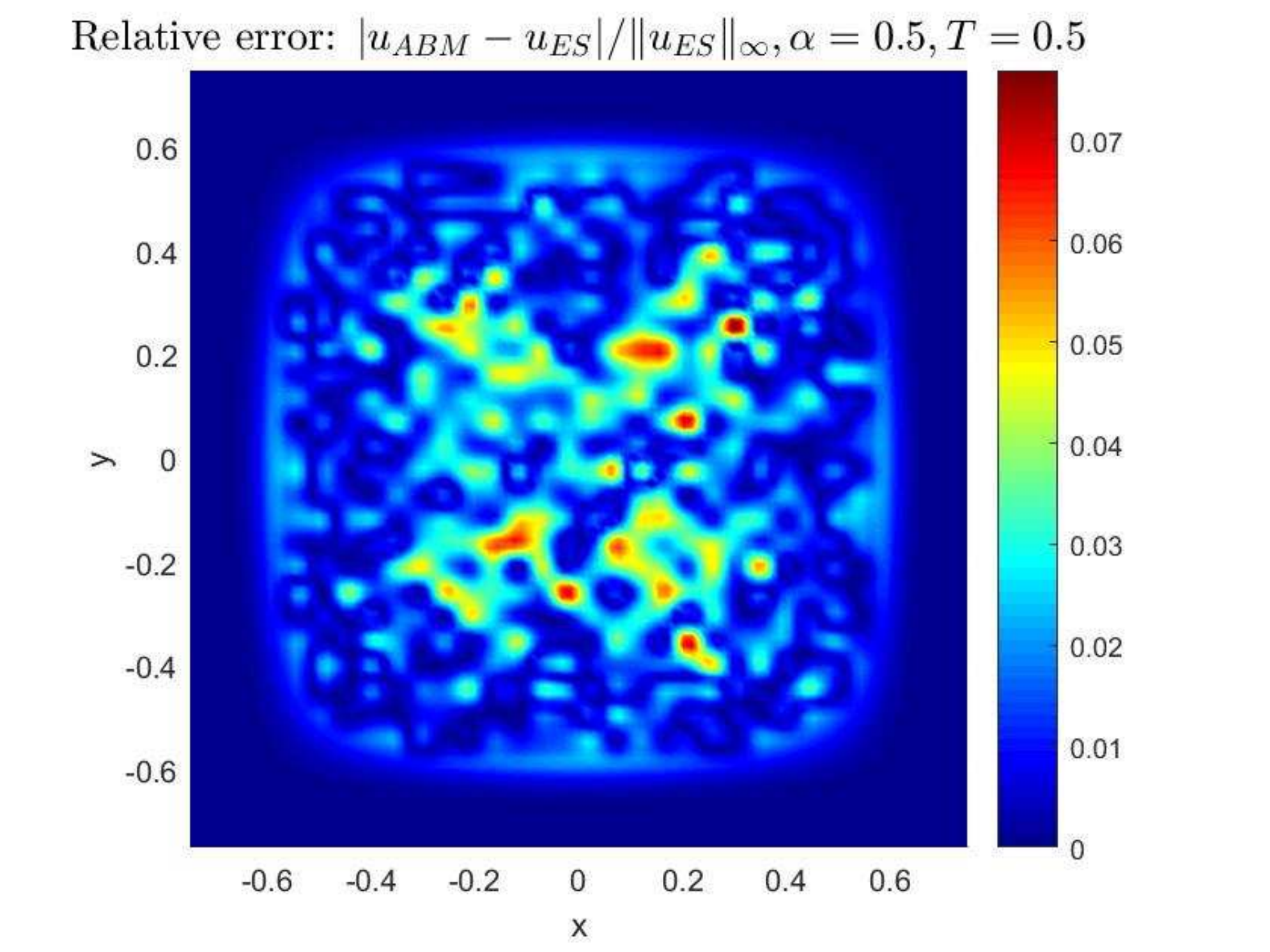}
        \caption{\scriptsize $T = 0.5$}
%        \label{fig:a00}
    \end{subfigure}
    \begin{subfigure}[b]{0.29\textwidth}
        \includegraphics[width=\textwidth]{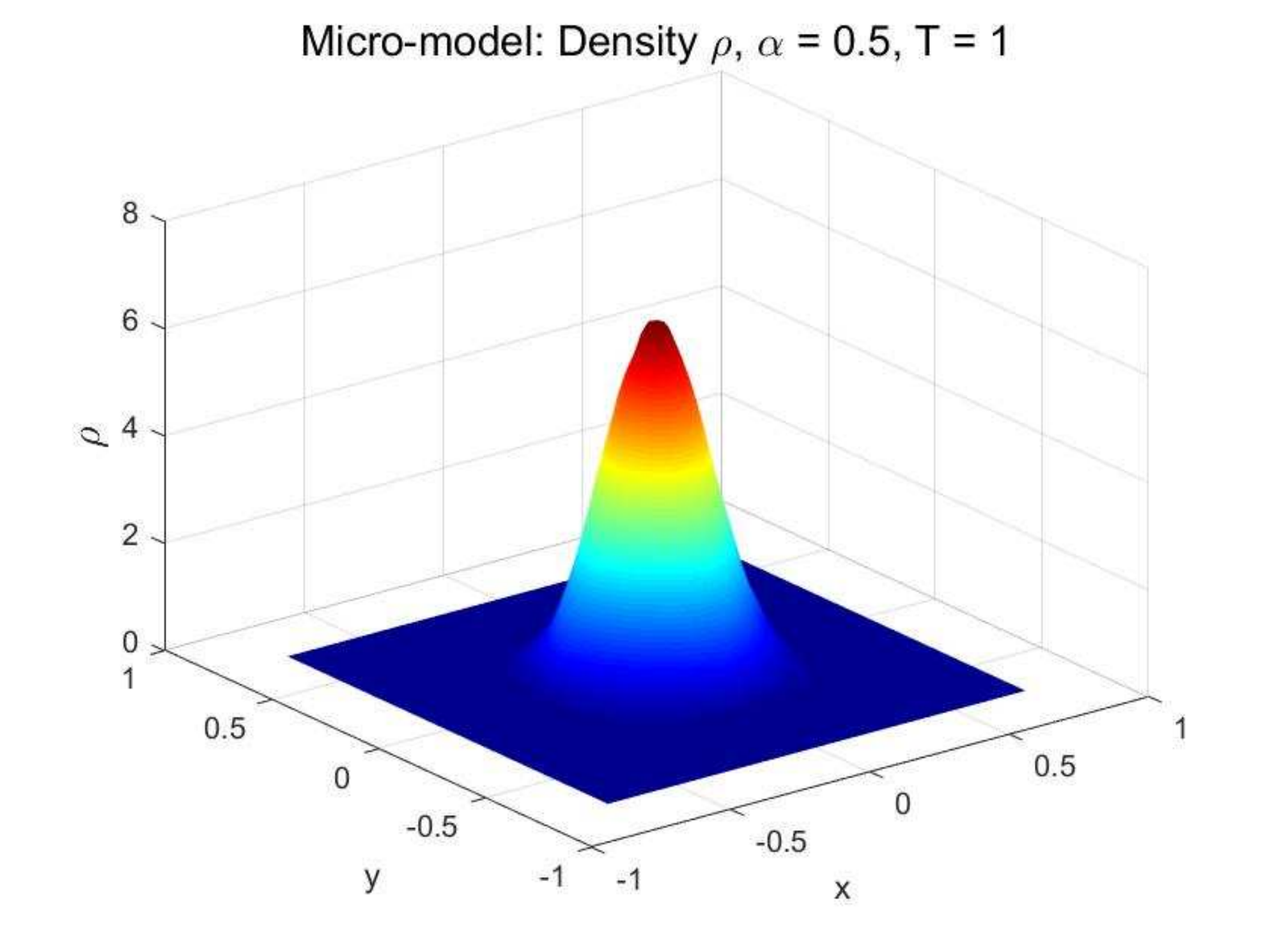}
        \includegraphics[width=\textwidth]{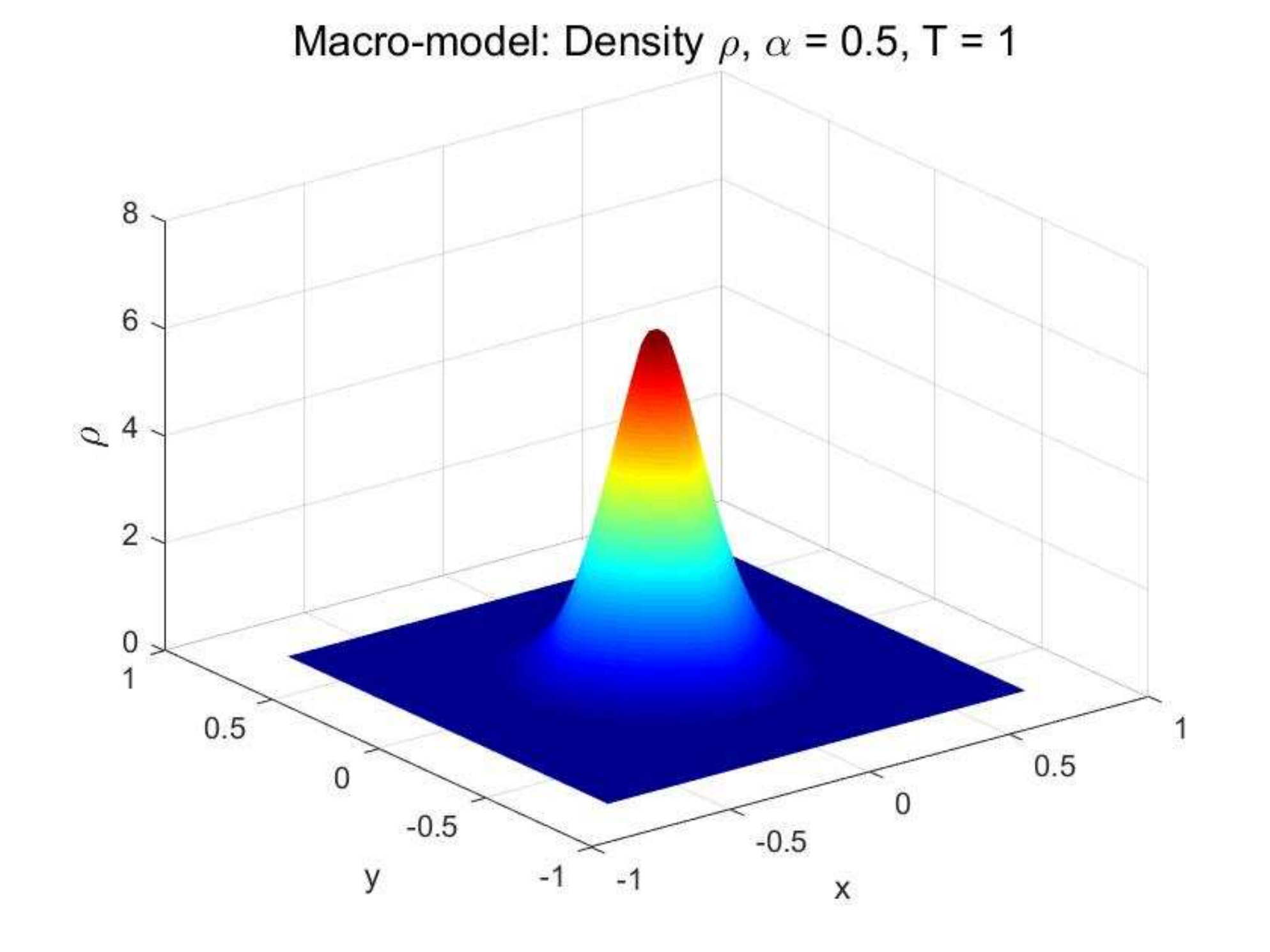}
        \includegraphics[width=\textwidth]{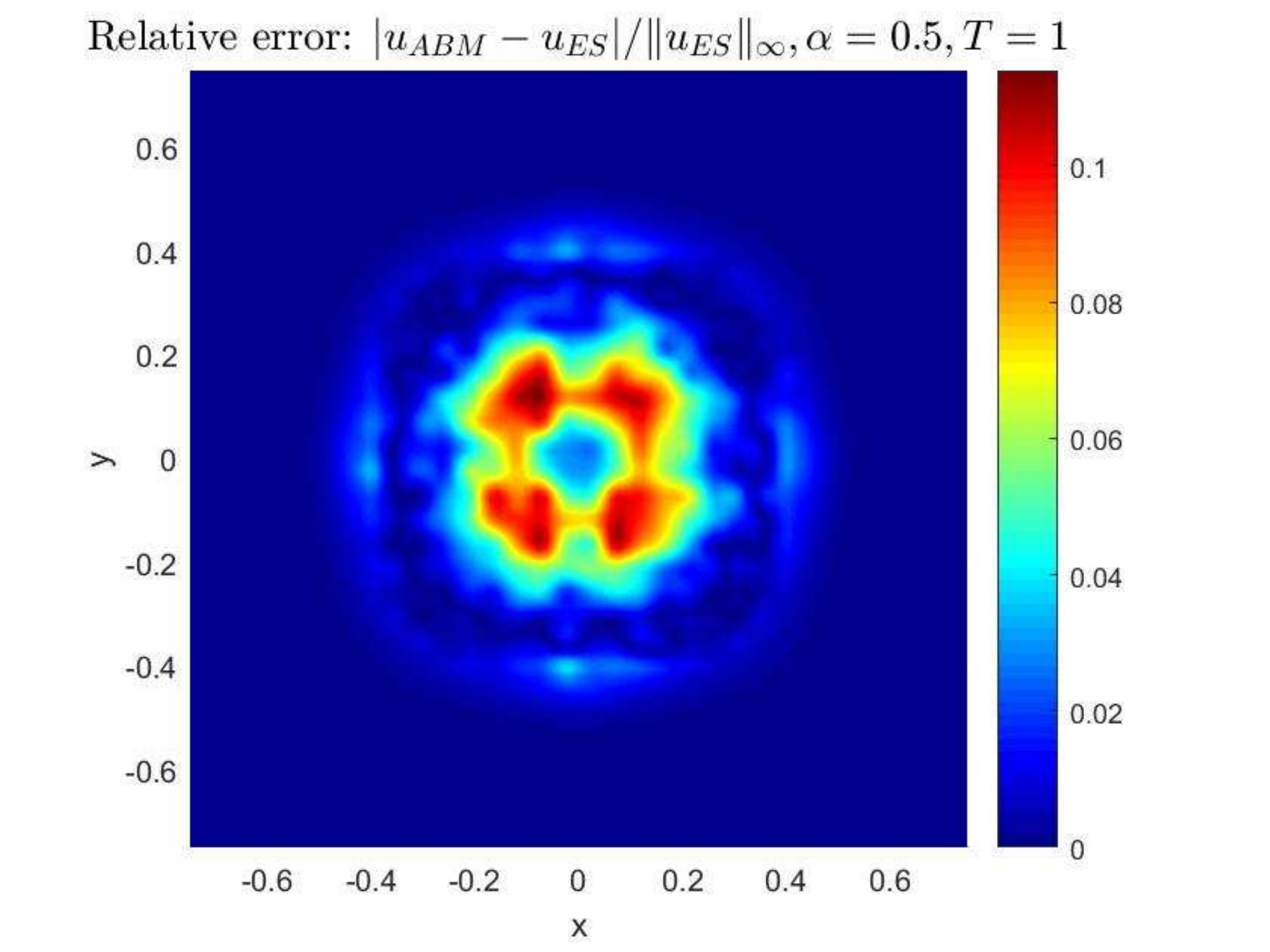}
        \caption{\scriptsize $T = 1.0$}
%        \label{fig:a01}
    \end{subfigure}
    %\vspace{12cm}
    \begin{subfigure}[b]{0.29\textwidth}
        \includegraphics[width=\textwidth]{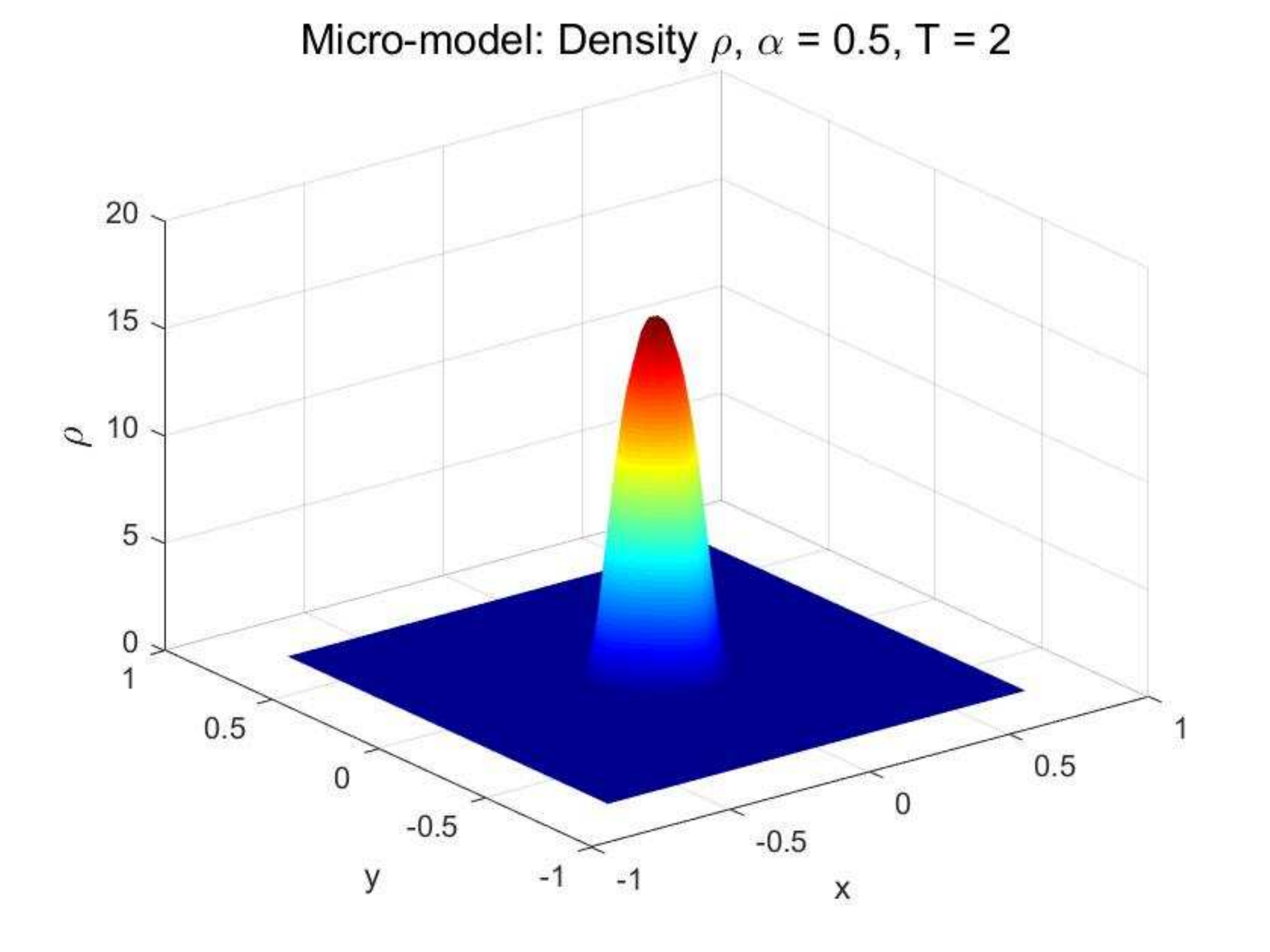}
        \includegraphics[width=\textwidth]{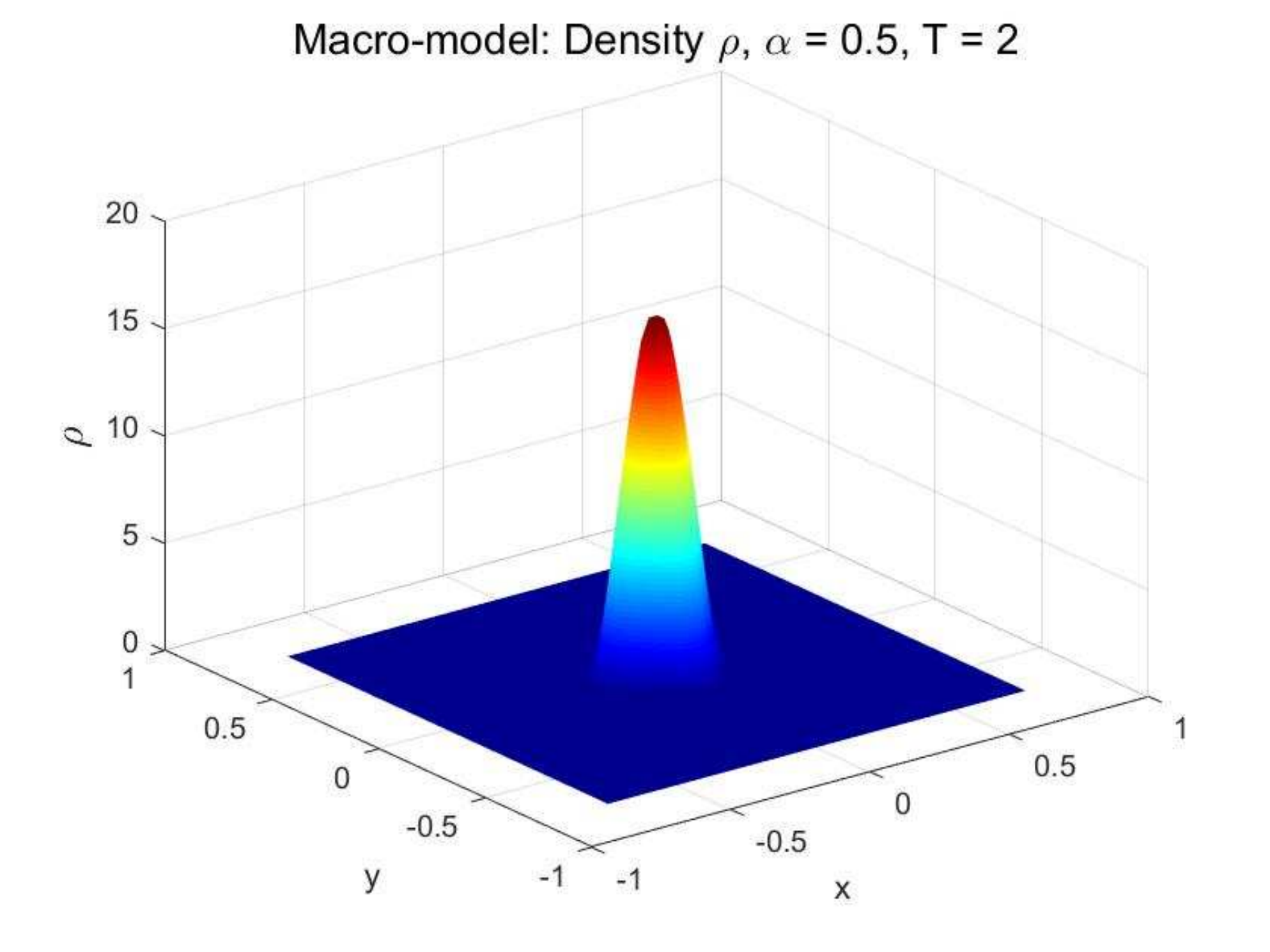}
        \includegraphics[width=\textwidth]{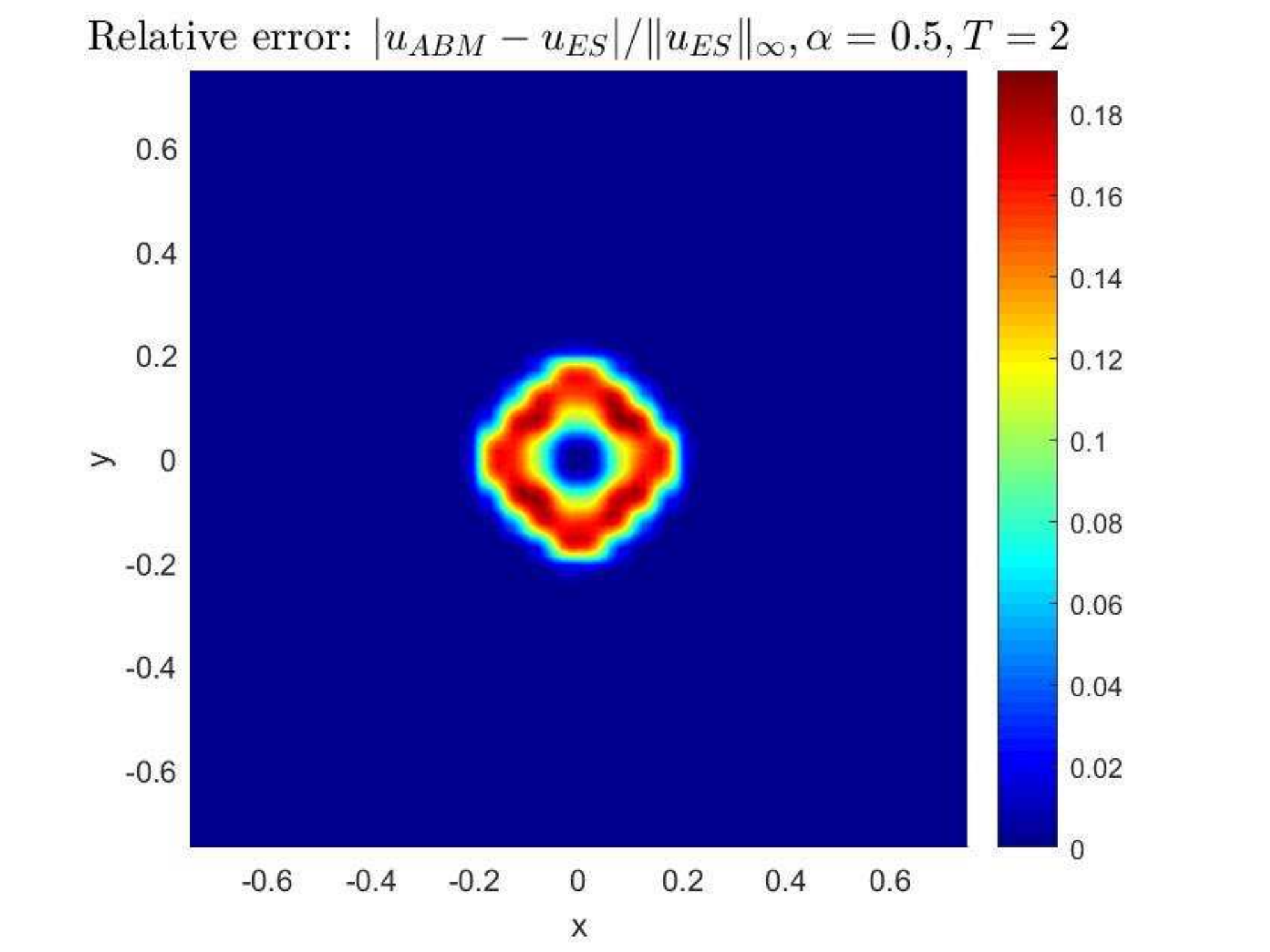}
        \caption{\scriptsize $T = 2.0$}
%        \label{fig:a01}
    \end{subfigure}
    \caption{Validation of two-dimensional case for the fractional order $\alpha = 0.5$: Numerical solutions of the density for the Microscale agent-based model \eqref{eq:Micro} (upper row) and the Macroscale Euler equations \eqref{Eulern} (middle row) at different times $T$.
    %Upper row: Micro-model solutions, middle row: Macro-model solutions,
    Lower row: relative error.}
\label{fig:Comparisondensity2Da05}
\end{figure}

\begin{figure}[t!]
    \centering
    \begin{subfigure}[b]{0.29\textwidth}
        \includegraphics[width=\textwidth]{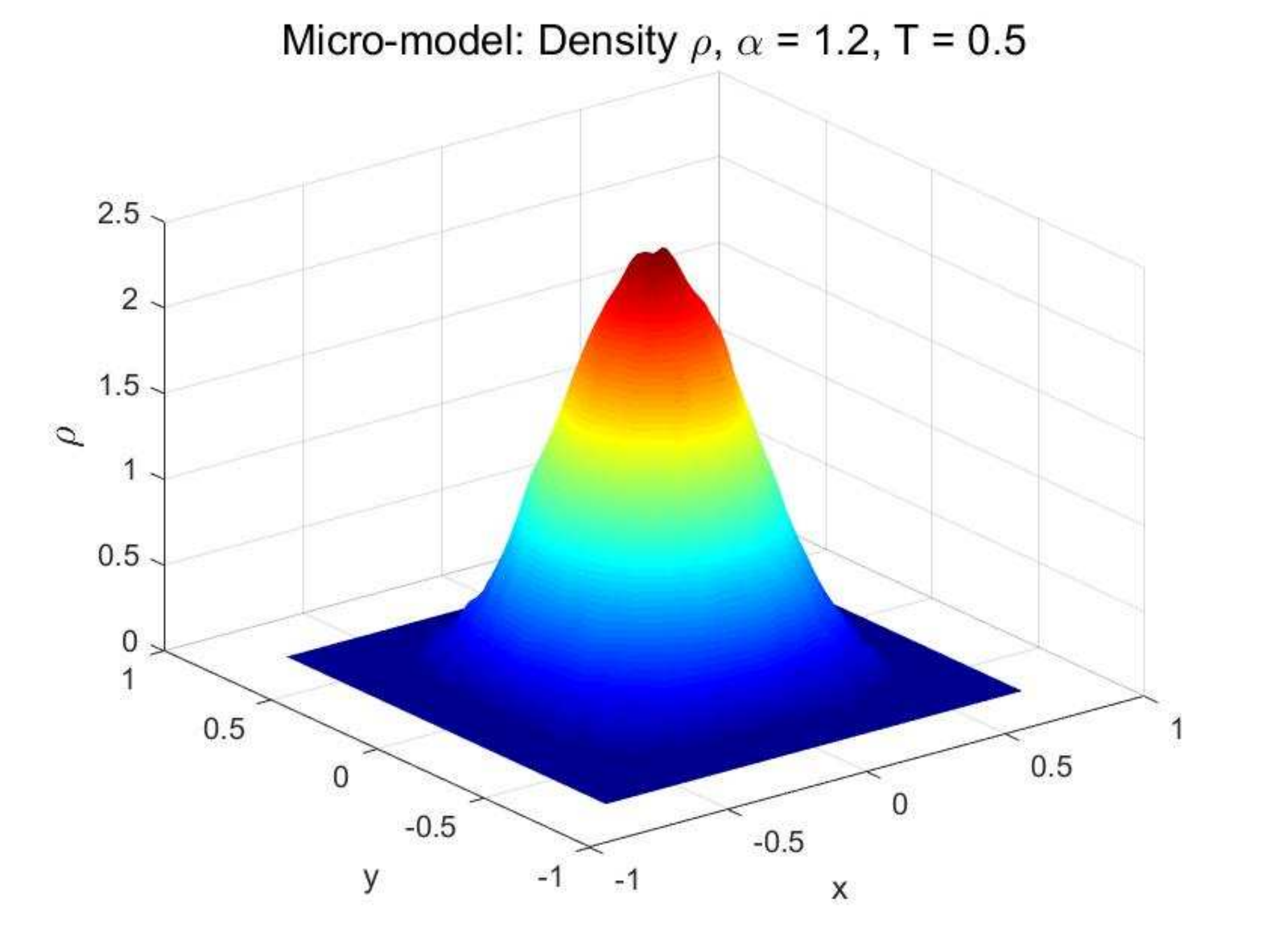}
        \includegraphics[width=\textwidth]{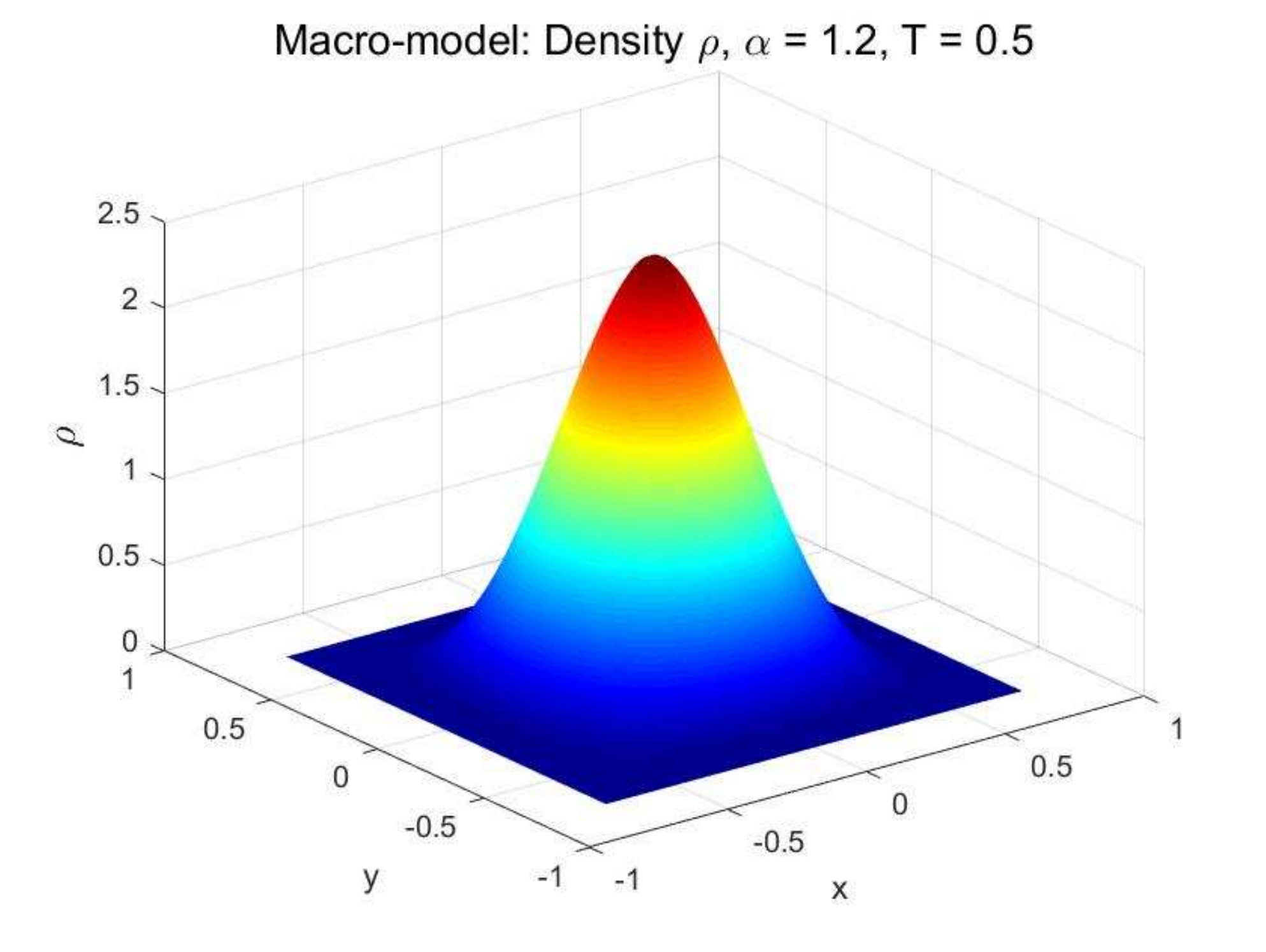}
        \includegraphics[width=\textwidth]{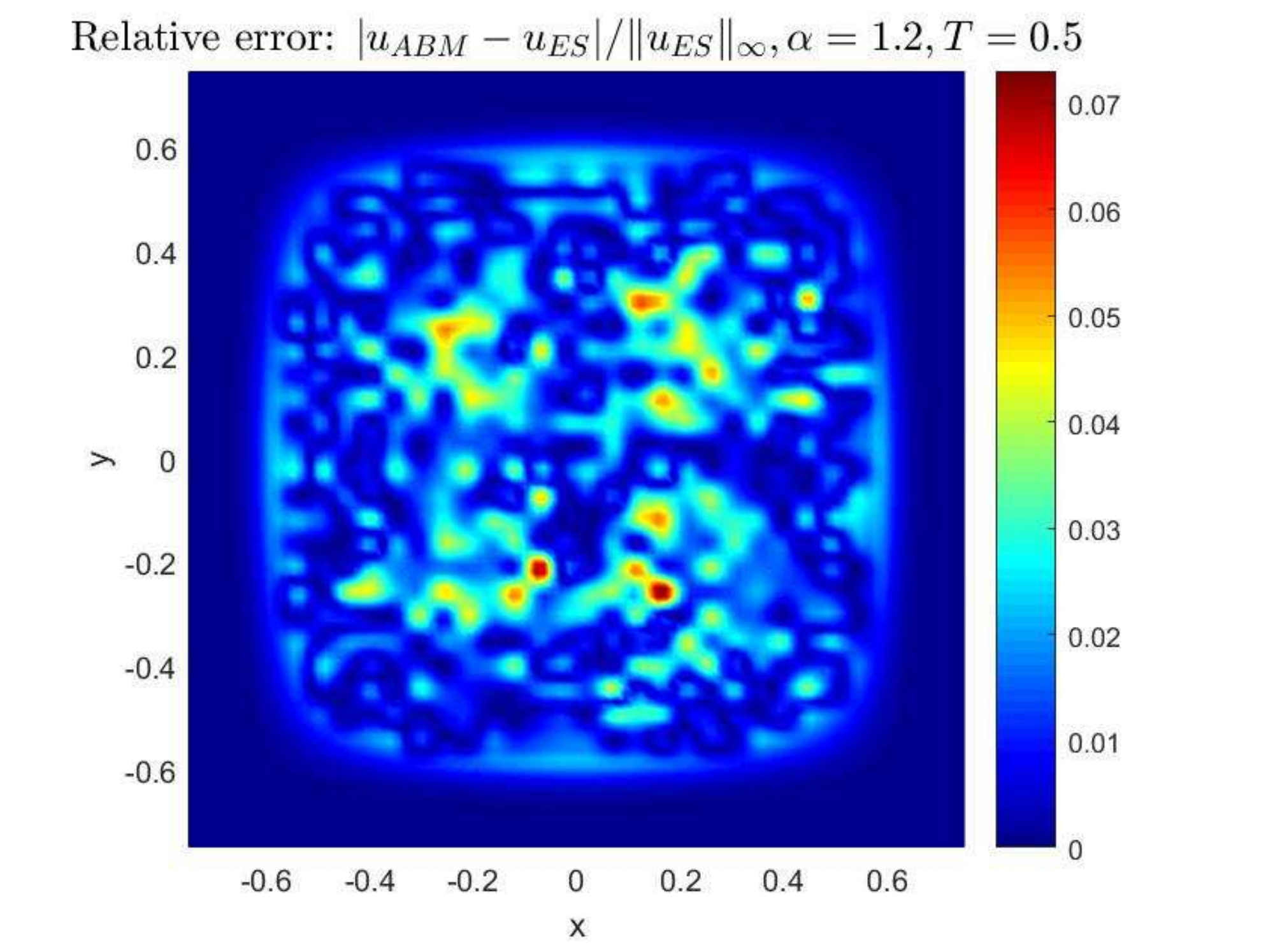}
        \caption{\scriptsize $T = 0.5$}
%        \label{fig:a00}
    \end{subfigure}
    \begin{subfigure}[b]{0.29\textwidth}
        \includegraphics[width=\textwidth]{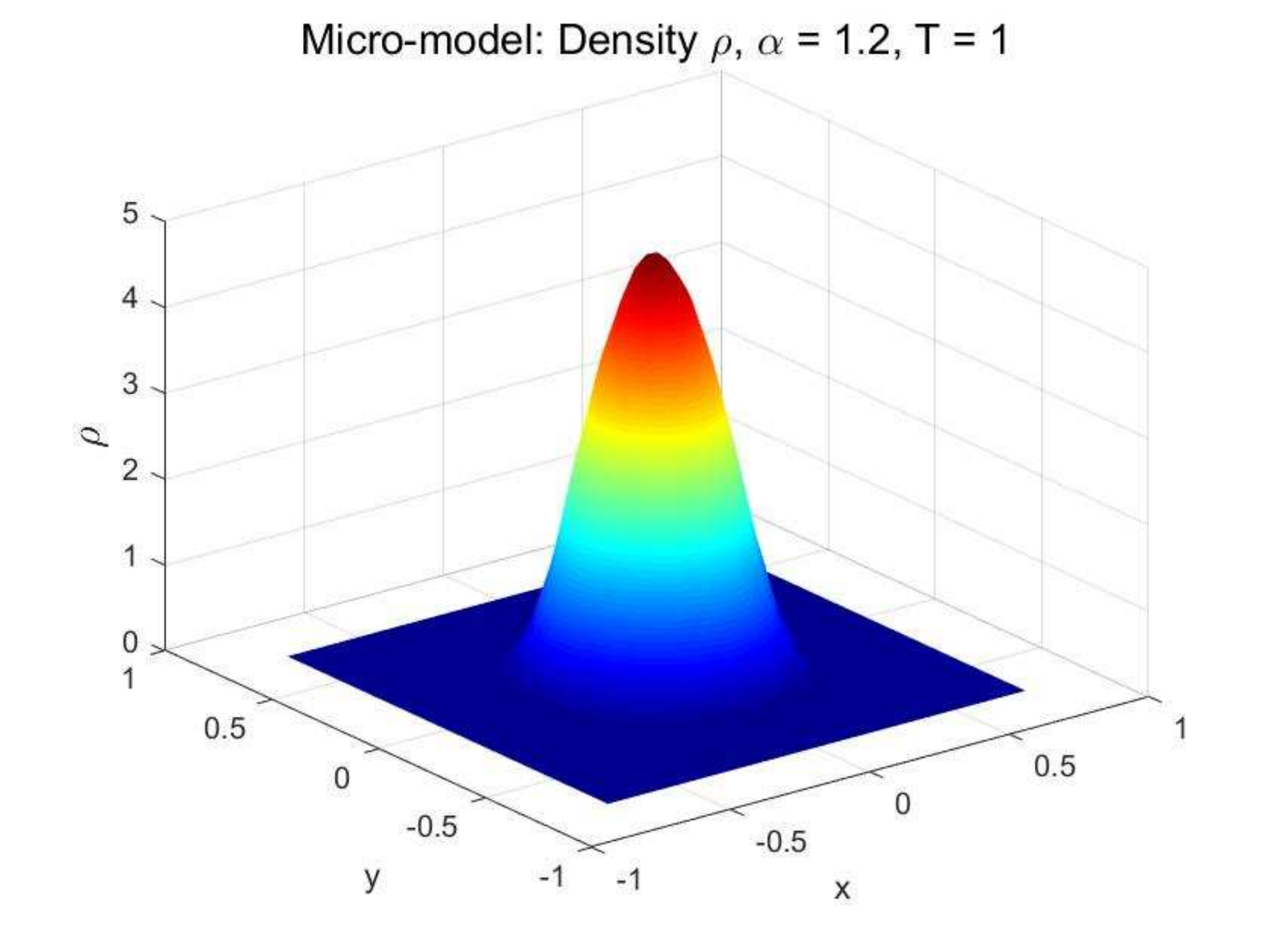}
        \includegraphics[width=\textwidth]{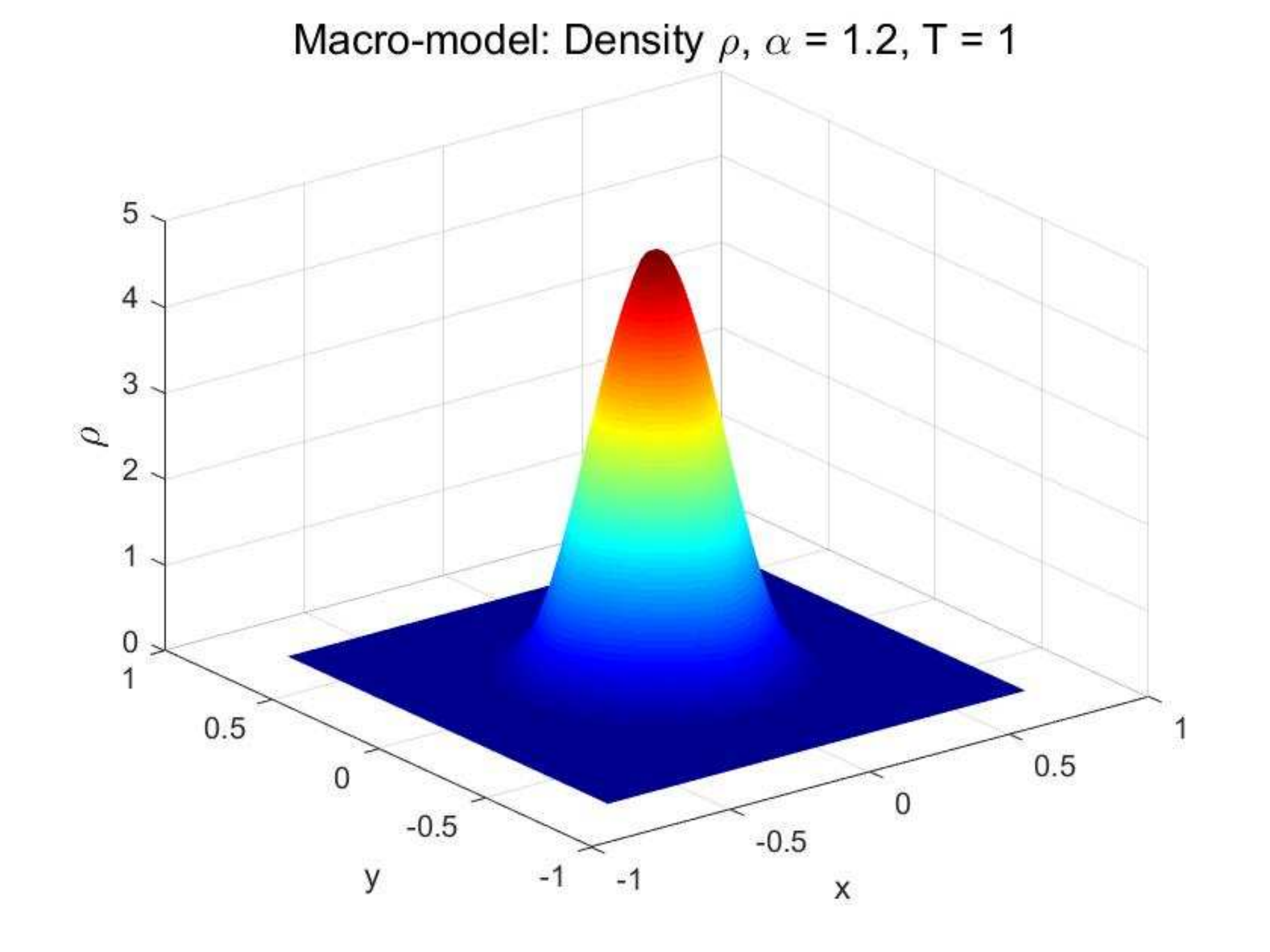}
        \includegraphics[width=\textwidth]{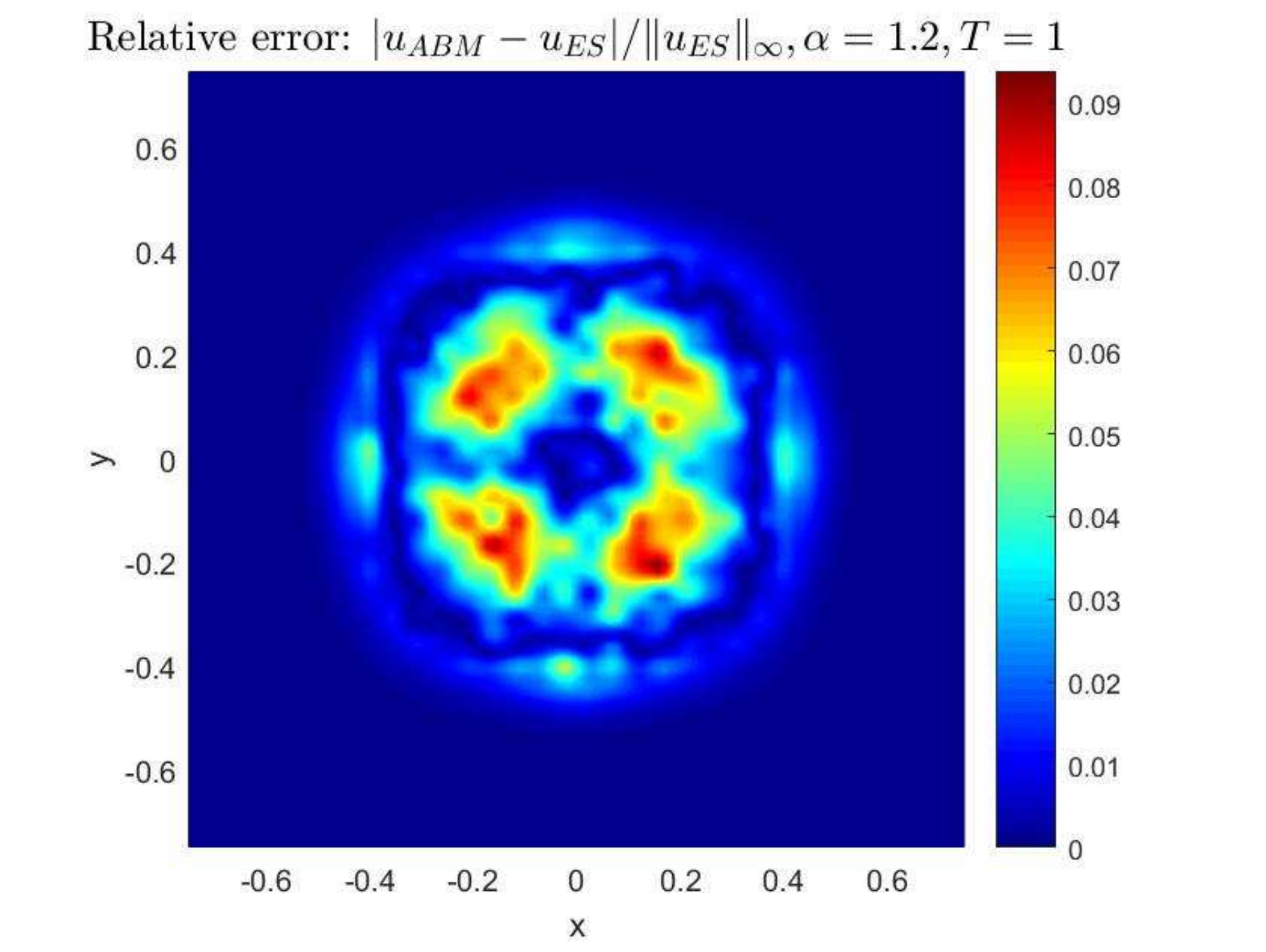}
        \caption{\scriptsize $T = 1.0$}
%        \label{fig:a01}
    \end{subfigure}
    %\vspace{12cm}
    \begin{subfigure}[b]{0.29\textwidth}
        \includegraphics[width=\textwidth]{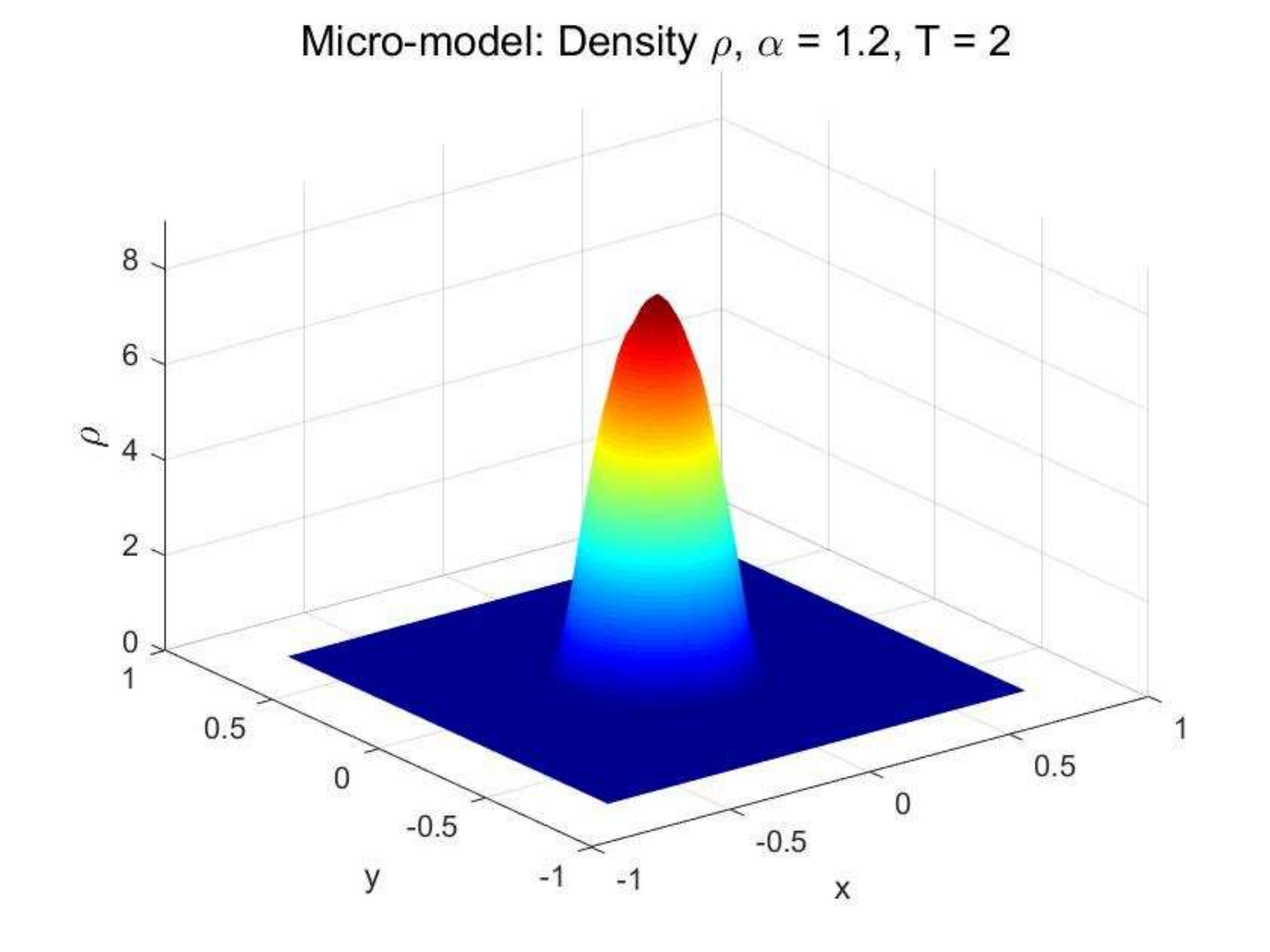}
        \includegraphics[width=\textwidth]{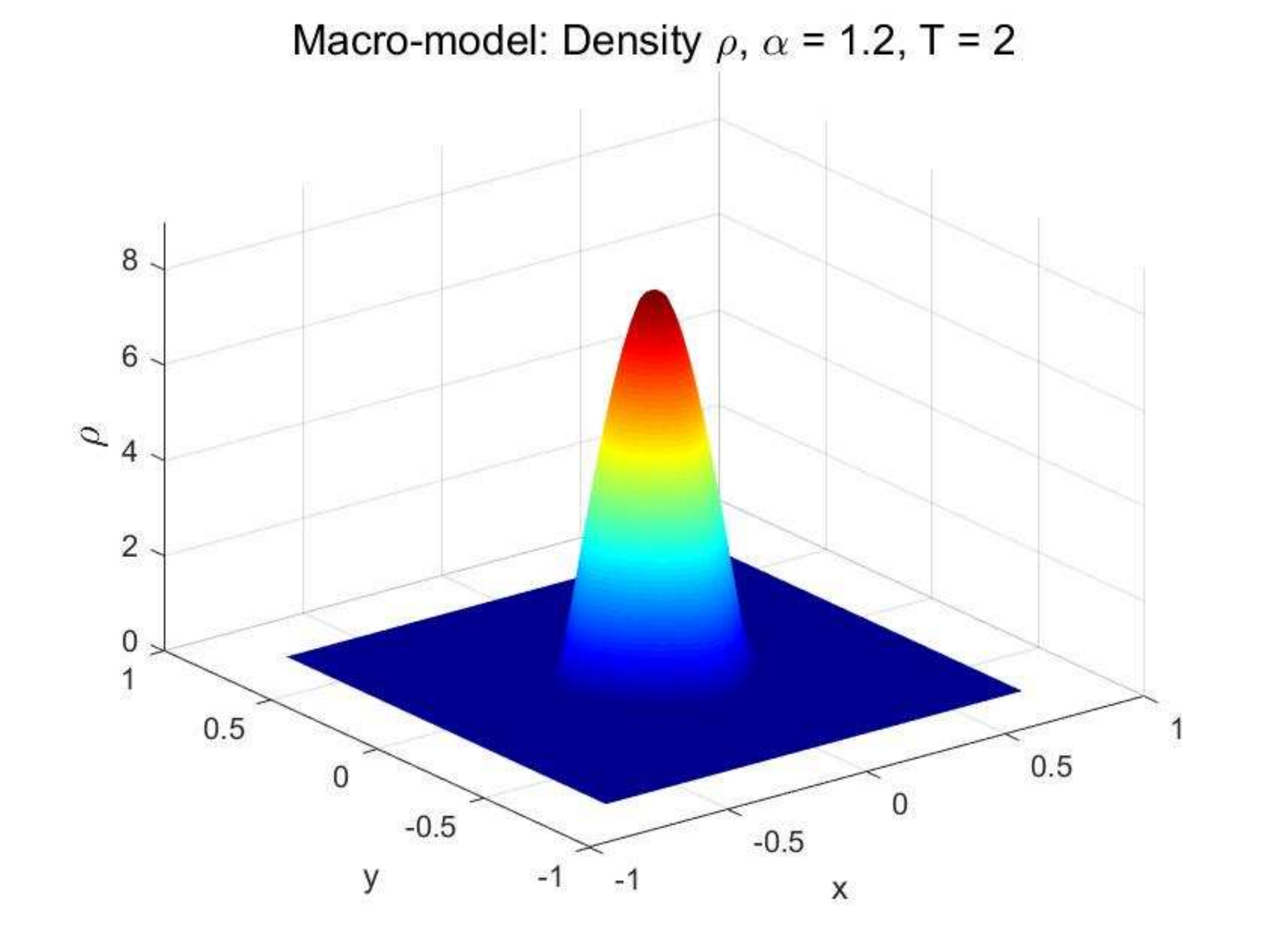}
        \includegraphics[width=\textwidth]{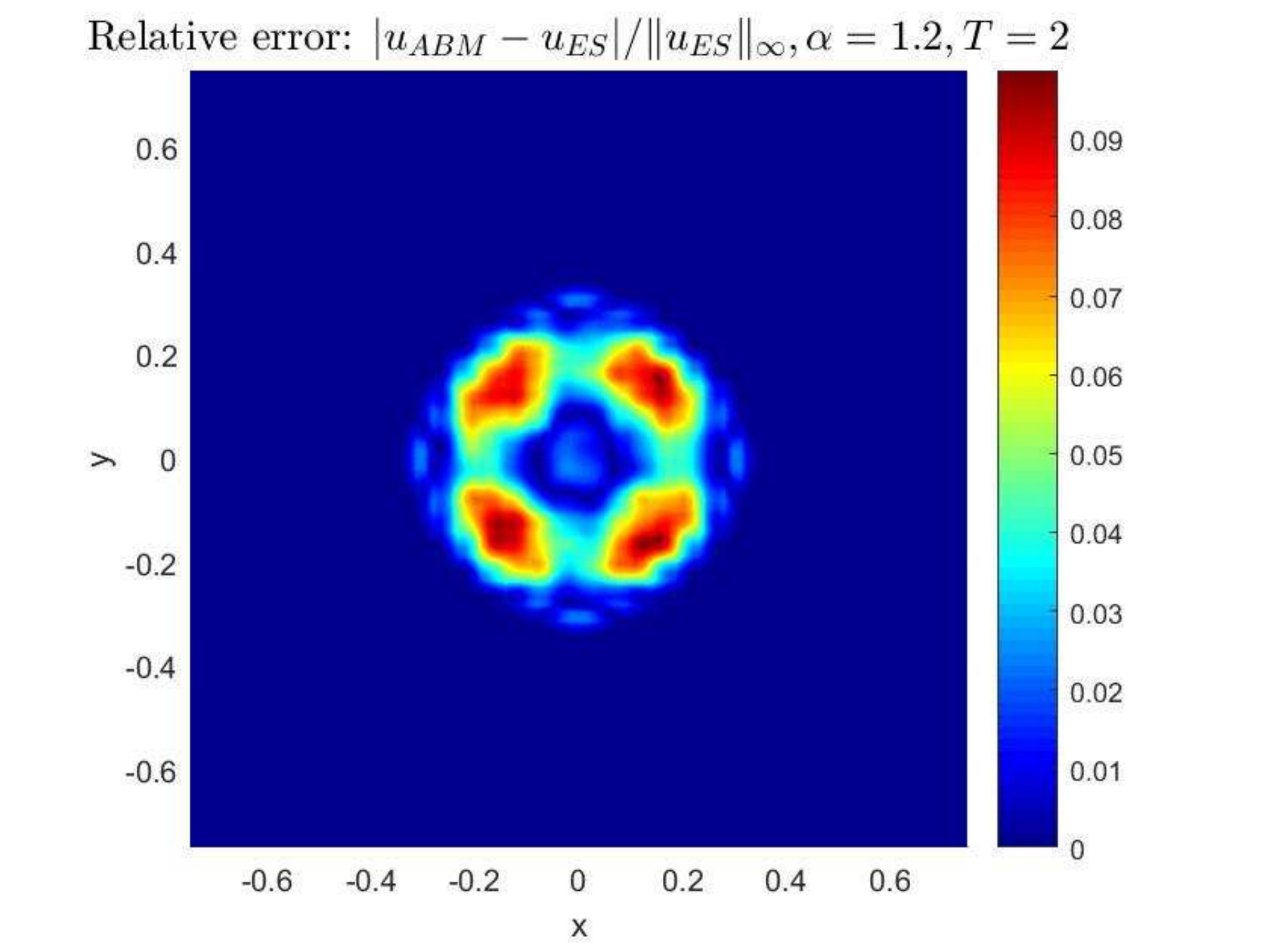}
        \caption{\scriptsize $T = 2.0$}
%        \label{fig:a01}
    \end{subfigure}
    \caption{Validation of two-dimensional case for the fractional order $\alpha = 1.2$: Numerical solutions of the density for the Microscale agent-based model \eqref{eq:Micro} (upper row) and the Macroscale Euler equations \eqref{Eulern} (middle row) at different times $T$.
    %Upper row: Micro-model solutions, middle row: Macro-model solutions,
    Lower row: relative error.}
\label{fig:Comparisondensity2Da12}
\end{figure}

\subsection{Gaussian process regression}
Assuming that we have an influence function with fractional order $\alpha$. We would like to infer the value of the fractional order $\alpha$. Suppose that we have the following data generated by the particle simulation:
\begin{equation*}
    Data = [Data_1,Data_2,\ldots, Data_N].
\end{equation*}
We then define the input-output pairs function for the Gaussian process as:
\begin{equation}\label{yfun}
     F(\alpha) = \frac{\|Num(\alpha) - Data\|}{\|Data\|},
\end{equation}
where the norm is in the $L_2$ sense, $Num(\alpha)$ is the numerical solution of the Euler equations  and $F(\alpha)$ is the scalar output corresponding to the fractional order $\alpha$.

In Gaussian process regression (GPR), we assume that $F(\alpha)$ is a Gaussian process:
$$F(\alpha) \sim GP(m(\alpha), k(\alpha,\alpha')),$$
where $F(\alpha)$ is a Gaussian random variable, $m(\cdot)$ and $k(\cdot, \cdot)$ are the mean and covariance functions, respectively.
GPR constructs the response surface based on the known input-output pairs, i.e., the training data. Assume that we have $N$ training data points given by
\begin{equation*}
    {\bm D} = \{(\alpha_1,F(\alpha_1)),(\alpha_2,F(\alpha_2)), \ldots, (\alpha_N,F(\alpha_N)) \}.
\end{equation*}
Generally, the output would contain noise, namely, the output would be
$$F(\alpha_i) + \epsilon_i, \quad i = 1,2,\ldots, N, $$
where the noise $\epsilon_i$ is considered to be Gaussian white noise, i.e., $\epsilon_N \sim \mathcal{N}(\mathbf{0},\sigma_n^2 \mathbf{I})$.

The goal of the GPR is to predict the output $F$ at arbitrary test input $\alpha^*$. This can be obtained by using the conditional distribution property for a multi-variate Gaussian random vector.
The joint distribution of the test output $F(\alpha^*)$ and the noisy training data can be written as $(N+1)$-variate Gaussian random vector
\begin{equation*}
\left(
  \begin{array}{c}
    F(\alpha^*) \\
    F(\alpha_1) \\
    \vdots \\
    F(\alpha_N) \\
  \end{array}
\right)
 \sim \mathcal{N}
 \left(
 \left[
 \begin{array}{c}
    m(\alpha^*) \\
    m(\alpha_1) \\
    \vdots \\
    m(\alpha_N) \\
  \end{array}
 \right]
 ,
\left[
  \begin{array}{cccc}
    k(\alpha^*,\alpha^*)&  k(\alpha^*,\alpha_1)               &   \cdots     &   k(\alpha^*,\alpha_N)    \\
     k(\alpha_1,\alpha^*)     &  k(\alpha_1,\alpha_1)+\sigma_n^2    &   \cdots        &  k(\alpha_1,\alpha_N) \\
    \vdots    & \vdots      & \ddots &    \vdots \\
    k(\alpha_N,\alpha^*) & k(\alpha_N,\alpha_1) & \cdots &  k(\alpha_N,\alpha_N)+\sigma_n^2 \\
  \end{array}
\right]
\right)
\end{equation*}
or in the matrix form
\begin{equation*}
\left(
  \begin{array}{c}
    F(\alpha^*) \\
    \mathbf{F}_N\\
  \end{array}
\right)
 \sim \mathcal{N}
 \left(
 \left[
 \begin{array}{c}
    m(\alpha^*) \\
    \mathbf{m}_N\\
  \end{array}
 \right]
 ,
\left[
  \begin{array}{cc}
    k(\alpha^*,\alpha^*)    &  \mathbf{k}^T  \\
    \mathbf{k}    &  \mathbf{K} + \sigma_n^2 \mathbf{I}\\
  \end{array}
\right]
\right).
\end{equation*}
The conditional distribution of $F(\alpha^*)$ given $\mathbf{F}_N$ is also a Gaussian distribution whose mean and covariance function are
\begin{equation*}
    \begin{aligned}
      &m_*(\alpha^*) = m(\alpha^*) + \mathbf{k}^T(\mathbf{K} + \sigma_n^2 \mathbf{I}) ^{-1}(\mathbf{F}_N - \mathbf{m}_N),\\
      &\sigma_*(\alpha^*) = k(\alpha^*,\alpha^*) - \mathbf{k}^T(\mathbf{K} + \sigma_n^2 \mathbf{I}) ^{-1} \mathbf{k}.
    \end{aligned}
\end{equation*}

The procedure of GPML mainly includes \emph{training, prediction} and \emph{Bayesian optimization}.
The training and prediction can be done with the GPR using the GPML Toolbox~\cite{2010rasmussen}. Next we introduce the Bayesian optimization.

\subsection{Bayesian optimization}
%The Gauss Process Regression can not only give the Prediction, but also can give the uncertainty of the predictions because it provides a predictive probability distribution of the test output for each test input.
The goal of Bayesian optimization in machine learning is to find the next input. Assume that we have $N$ training data $\alpha_1,\alpha_2,\ldots, \alpha_N$. To obtain the next input data $\alpha_{N+1}$, we need to optimize the acquisition function, i.e.,
\begin{equation}\label{maxacq}
    \alpha_{N+1} = \arg\max_{\alpha \in \Omega} \{EI(\alpha)\},
\end{equation}
where $EI(\cdot)$ is the expected improvement acquisition function defined by
\begin{equation*}\label{acq}
    EI(\alpha) = [\min(\mathbf{F}_N) - m_{*}(\alpha)]\Phi\left(\frac{\min(\mathbf{F}_N) - m_{*}(\alpha)}{\sigma_{*}(\alpha)} \right) + \sigma_{*}(\alpha)\phi\left(\frac{\min(\mathbf{F}_N) - m_{*}(\alpha)}{\sigma_{*}(\alpha)} \right),
\end{equation*}
here $\Phi(\cdot)$ and $\phi(\cdot)$ are the standard normal cumulative distribution and density functions respectively, $m_{*}(\cdot)$ and $\sigma_{*}(\cdot)$ are the predicted mean and standard deviation, respectively; $\mathbf{F}_N$ is the vector consisting of the $N$ training outputs.
The optimization problem \eqref{maxacq} is solved by using the modified Lipschitzian optimization algorithm~\cite{1993Jones}.
Once we obtain the new training input $\alpha_{N+1}$ with the Bayesian optimization, we then solve the Euler system of equations  to get the numerical density and velocity, and then obtain the output $F(\alpha_{N+1})$ by the equation \eqref{yfun}.
Thus, a new training data set $(\alpha_{N+1},F(\alpha_{N+1}))$ is obtained and the next iteration can start until the terminal condition is satisfied.

\subsection{Numerical examples}
We now consider two numerical examples to illustrate the GPML algorithm.
In the following numerical examples, the covariance function $k(\bm {x}, \bm{x}')$ is taken as the frequently used Matern $5/2$-order automatic relevance determination function~\cite{2004rasmussen}
$$k(\bm {x}, \bm{x}') = \sigma^2\left(1 + h + \frac{h^2}{3}\right)e^{-h},$$
where
$$h = \sqrt{\left(\sum_{i = 1}^d \frac{5(x_i - x_i')^2}{\theta_i^2}\right)}$$
with $d$ the dimension of the input vector $\bm{x}$ and $\theta_i$ the $i$-th hyper-parameter.

\subsubsection{One-dimensional case}
\begin{exam}\label{ex:1D:ML}
We first consider the one-dimensional problem with the same influence function and initial and boundary conditions as that for Example \ref{ex:1D:sln}.
\end{exam}

We first obtain the data by solving the agent-based model \eqref{eq:Micro} with the particle method for a given $\hat{\alpha}$. We then get the positions and the corresponding velocities $v_i^{\hat{\alpha},P}(t), i = 1,\ldots,N$ of particles $x_i^{\hat{\alpha}}(t), i = 1,\ldots,N$ at time $t = k \Delta T$, where $\Delta T = 0.1, k = 5,6,\ldots, 20$, here we use $N = 1024$ as used in Example \ref{ex:1D:sln}.
Therefore, given an input data $\alpha$, we obtain the output data as follows: Firstly, solving the Euler equations  \eqref{Eulern} with the finite volume method to obtain the numerical velocity $v_i^{\alpha,Num}(t)$ at the given points $x_i^{\hat{\alpha}}(t), i = 1,\ldots,N$, then  the output is given by
\begin{equation*}
    F(\alpha) = \frac{\|{\bm v}_{\alpha,Num} -{\bm v}_{\hat{\alpha},P}\|}{\|{\bm v}_{\hat{\alpha},P}\|},
\end{equation*}
where
$$ {\bm v}_{\hat{\alpha},P}: = \left(v_1^{\hat{\alpha},P}(t),\ldots, v_N^{\hat{\alpha},P}(t)\right)^T,\; {\bm v}_{\alpha,Num}: = \left(v_1^{\alpha,Num}(t),\ldots, v_N^{\alpha,Num}(t)\right)^T.$$
We generate the data using agent-based simulation with two different values of the fractional order, given by $\hat{\alpha} = 0.5, 1.2$, and learn the effective fractional order $\alpha$ for the Euler equations by using the GPML algorithm.
To solve the Euler equations  \eqref{Eulern}, we use $\Delta x = 1/256$.

The result of learning the value of $\alpha$ is given in Table \ref{tab:s0512}. For the given data generated by the agent-based model with $\hat{\alpha}$, Table \ref{tab:s0512} shows that we can successfully use the Euler equations  to infer the fractional order $\hat{\alpha}$ with the GPML algorithm. The relative error between the agent-based system and the learned Euler system of equations  is about $1\%$ for the one-dimensional cases.
We also show the convergence of the Bayesian optimization in Figure \ref{fig:convergences0512}. Observe that the learning process converges after about 20 iterations.
We observe from the result that we can use the macroscale Euler equations  to learn the effective non-local influence function, namely, learn the microscale agent-based model directly from particle trajectories generated by the agent-based simulations.

\begin{table}[!h]
\begin{center}
\begin{tabular}{|c|c|c|}
\hline  The given  $\hat{\alpha}$ & Learned value of $\alpha$ & Output $F(s)$ \\
\hline  $0.5$ & 0.4803  & 1.2174e-02 \\
\hline  $1.2$ & 1.1651  & 7.9663e-03 \\
\hline
\end{tabular}
\end{center}
\caption{Learned value of $\alpha$ and relative error of the mean field $F(s)$ for the one-dimensional case.
}\label{tab:s0512}
\end{table}

%\begin{table}[!h]
%\begin{center}
%\begin{tabular}{|c|c|c|c|}
%\hline  The given  $\hat{\alpha}$ & Learned value of $\alpha$ & Output $F(s)$& Deviation  \\
%\hline  $0.5$ & 0.4803  & 1.2174e-02 &9.9121e-08 \\
%\hline  $1.2$ & 1.1651  & 7.9663e-03 &2.1221e-04 \\
%\hline
%\end{tabular}
%\end{center}
%\caption{Learned mean value of $\alpha$ and relative error of the mean field $F(s)$ for the one-dimensional case.
%}\label{tab:s0512}
%\end{table}

%
\begin{figure}[htp]
\begin{center}
\begin{minipage}{0.46\linewidth}
\begin{center}
\includegraphics[scale=0.50,angle=0]{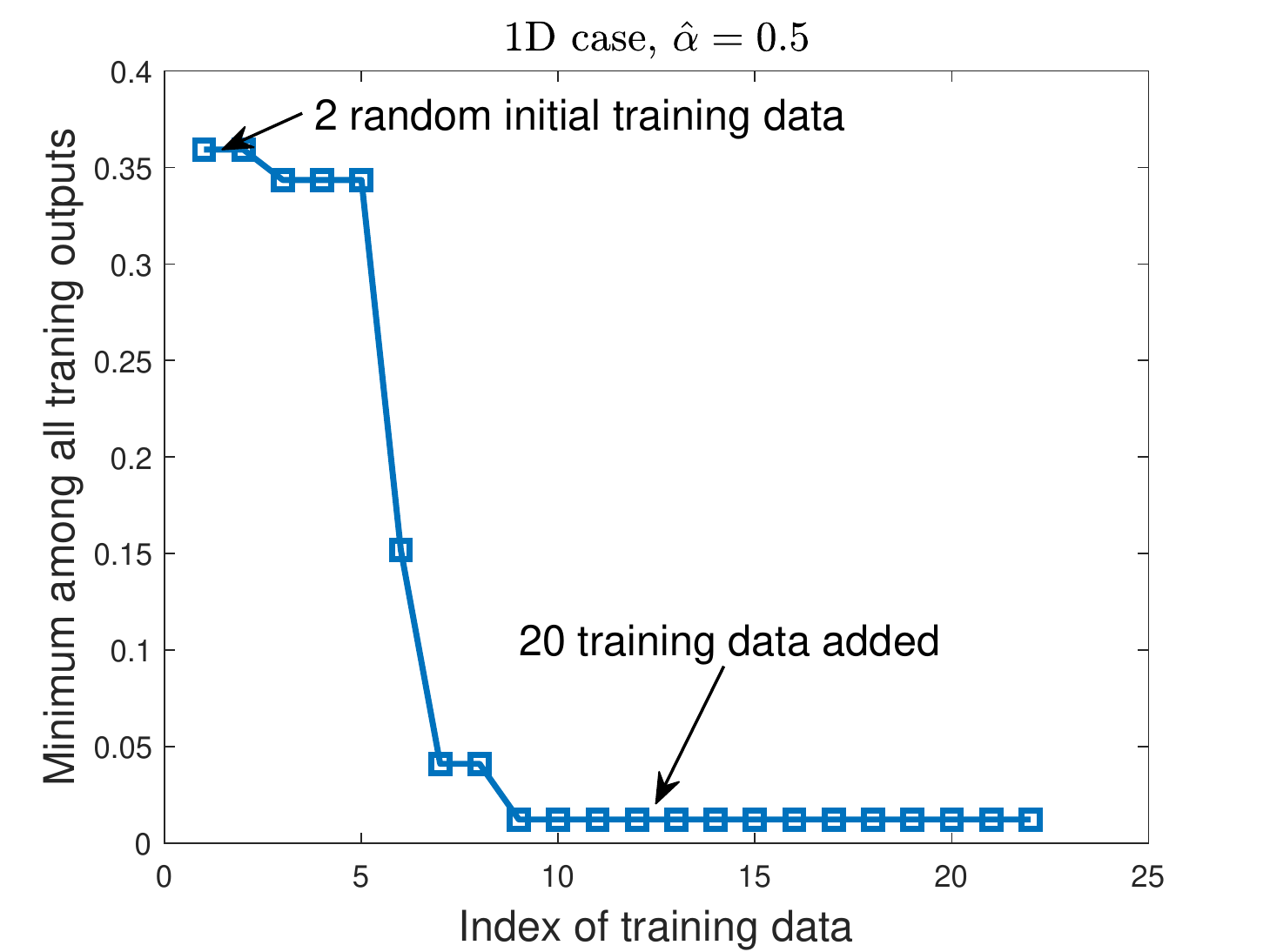}\end{center}
\end{minipage}
\begin{minipage}{0.46\linewidth}
\begin{center}
\includegraphics[scale=0.50,angle=0]{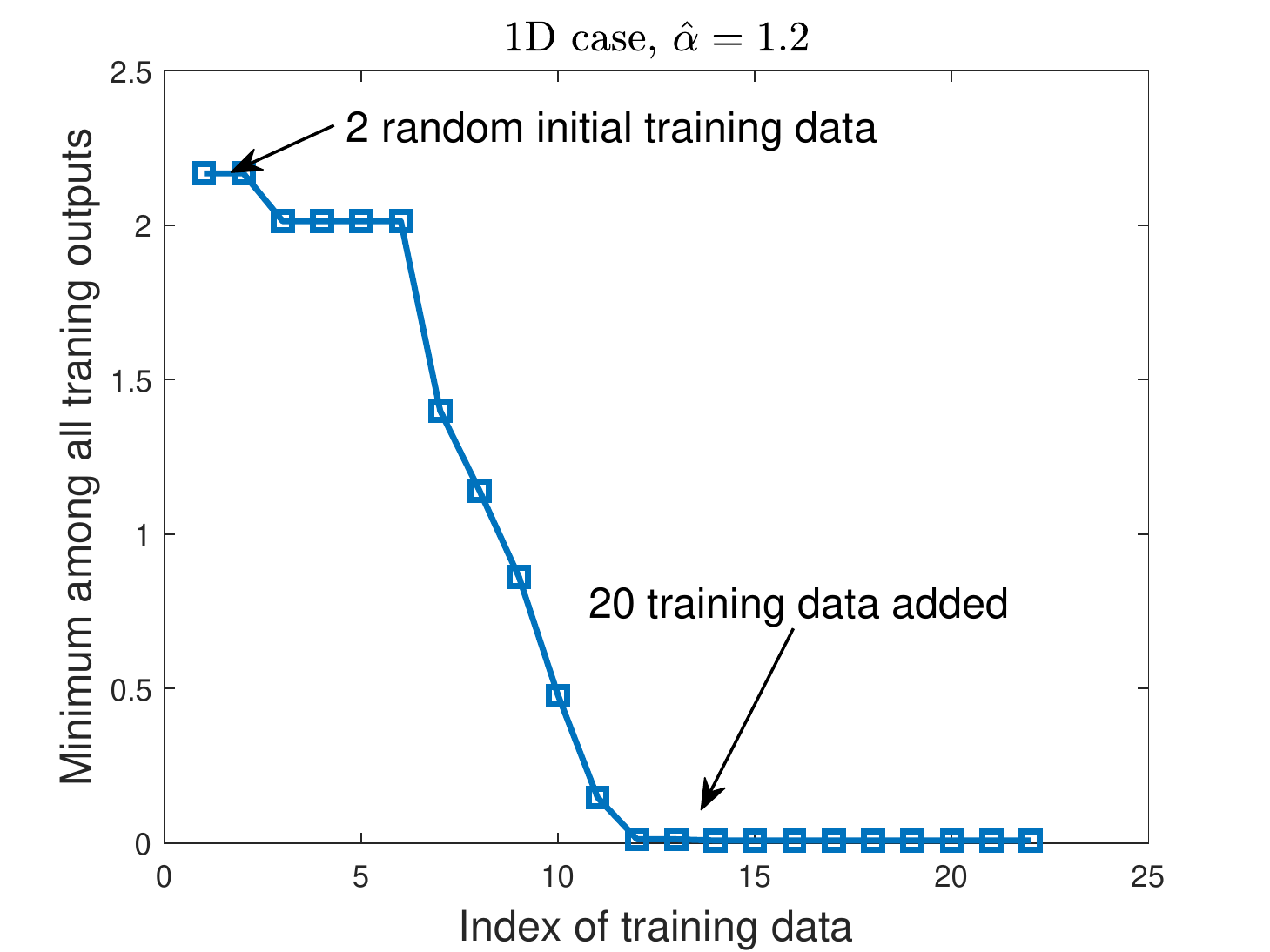}\end{center}
\end{minipage}
\caption{Convergence of the optimization problem for the one-dimensional case: Left: $\hat{\alpha} = 0.5$, right: $\hat{\alpha} = 1.2$.
}\label{fig:convergences0512}
\end{center}
\end{figure}

\subsubsection{Two-dimensional case}

\begin{exam}\label{ex:2D:ML}
We now consider the two-dimensional problem with the same influence function and initial and boundary conditions as that for Example \ref{ex:2D:sln}.
\end{exam}
Similarly as the one-dimensional case,
we obtain the data by solving the agent-based model \eqref{eq:Micro} with the particle method for a given $\hat{\alpha}$ in the squire domain $\Omega$ generating the positions $(x_i^{\hat{\alpha}}(t),y_i^{\hat{\alpha}}(t)), \, i = 1,\ldots,N$ and the corresponding velocities $u_i^{\hat{\alpha},P}(t)$ and $v_i^{\hat{\alpha},P}(t), \,i = 1,\ldots, N$ of the particles at time $t = k \Delta T$, where $\Delta T = 0.1, k = 5,6,\ldots, 20$, here we set $N = 9976$.

Therefore, for a given input data $\alpha$, we obtain the output data as follows: Firstly, solving the Euler equations  \eqref{Eulern} with the finite volume method to obtain the numerical velocity $u_i^{\alpha,Num}(t)$ and $v_i^{\alpha,Num}(t)$ at given points $(x_i^{\hat{\alpha}}(t), y_i^{\hat{\alpha}}(t)), i = 1,\ldots,N$,
then the output is given by
\begin{equation*}
    F(\alpha) = \frac{1}{2}\left(\frac{\|{\bm u}_{s,Num} -{\bm u}_{\hat{\alpha},P}\|}{\|{\bm u}_{\hat{\alpha},P}\|}
    + \frac{\|{\bm v}_{\alpha,Num} -{\bm v}_{\hat{\alpha},P}\|}{\|{\bm v}_{\hat{\alpha},P}\|}\right),
\end{equation*}
where
\begin{equation*}
    \begin{aligned}
       &{\bm u}_{\hat{\alpha},P}: = \left(u_1^{\hat{\alpha},P}(t),\ldots, u_{N}^{\hat{\alpha},P}(t)\right)^T,\; {\bm u}_{\alpha,Num}: = \left(u_1^{\alpha,Num}(t),\ldots, u_{N}^{\alpha,Num}(t)\right)^T, \\
       &{\bm v}_{\hat{\alpha},P}: = \left(v_1^{\hat{\alpha},P}(t),\ldots, v_{N}^{\hat{\alpha},P}(t)\right)^T,\; {\bm v}_{\alpha,Num}: = \left(v_1^{\alpha,Num}(t),\ldots, v_{N}^{\alpha,Num}(t)\right)^T.
    \end{aligned}
\end{equation*}

In the numerical simulations, we test two different agent-based systems with $\hat{\alpha} = 0.5, 1.2$.
To solve the Euler equations  \eqref{Eulern}, the space steps of the finite volume method are $\Delta x = \Delta y = 1/64$.
Initially, we generate the training data with two random values of $\alpha$.
The learned values of $\alpha$ are shown in Table \ref{tab:s05122D}, and the convergence results of the Bayesian optimization are shown in Figure \ref{fig:convergence2Ds0512}. Again, we observe that we can infer the influence function by solving the Euler system of equations  with the data generated by the agent-based model. The relative error between the agent-based system and the learned Euler system of equations  is about $2\%$ for the two-dimensional cases.

\begin{table}[t!]
\begin{center}
\begin{tabular}{|c|c|c|c|}
\hline  The given  $\hat{\alpha}$ & Learned value of $\alpha$ & Output $F(\alpha)$ \\
\hline  $0.5$ & 0.5134  & 2.1428e-02 \\
\hline  $1.2$ & 1.2009  & 2.0233e-02 \\
\hline
\end{tabular}
\end{center}
\caption{Learned value of $\alpha$ and relative error of the mean field $F(s)$ for the two-dimensional case.
}\label{tab:s05122D}
\end{table}

%\begin{table}[t!]
%\begin{center}
%\begin{tabular}{|c|c|c|c|}
%\hline  The given  $\hat{\alpha}$ & Learned mean value of $\alpha$ & Output $F(\alpha)$& Deviation \\
%\hline  $0.5$ & 5.1339  & 2.1428e-02  & 3.5242e-07\\
%\hline  $1.2$ & 1.2009  & 2.0233e-02  & 1.7274e-07\\
%\hline
%\end{tabular}
%\end{center}
%\caption{Learned mean value of $\alpha$ and relative error of the mean field $F(s)$ for the two-dimensional case.
%}\label{tab:s05122D}
%\end{table}

%
\begin{figure}[t!]
\begin{center}
\begin{minipage}{0.46\linewidth}
\begin{center}
\includegraphics[scale=0.50,angle=0]{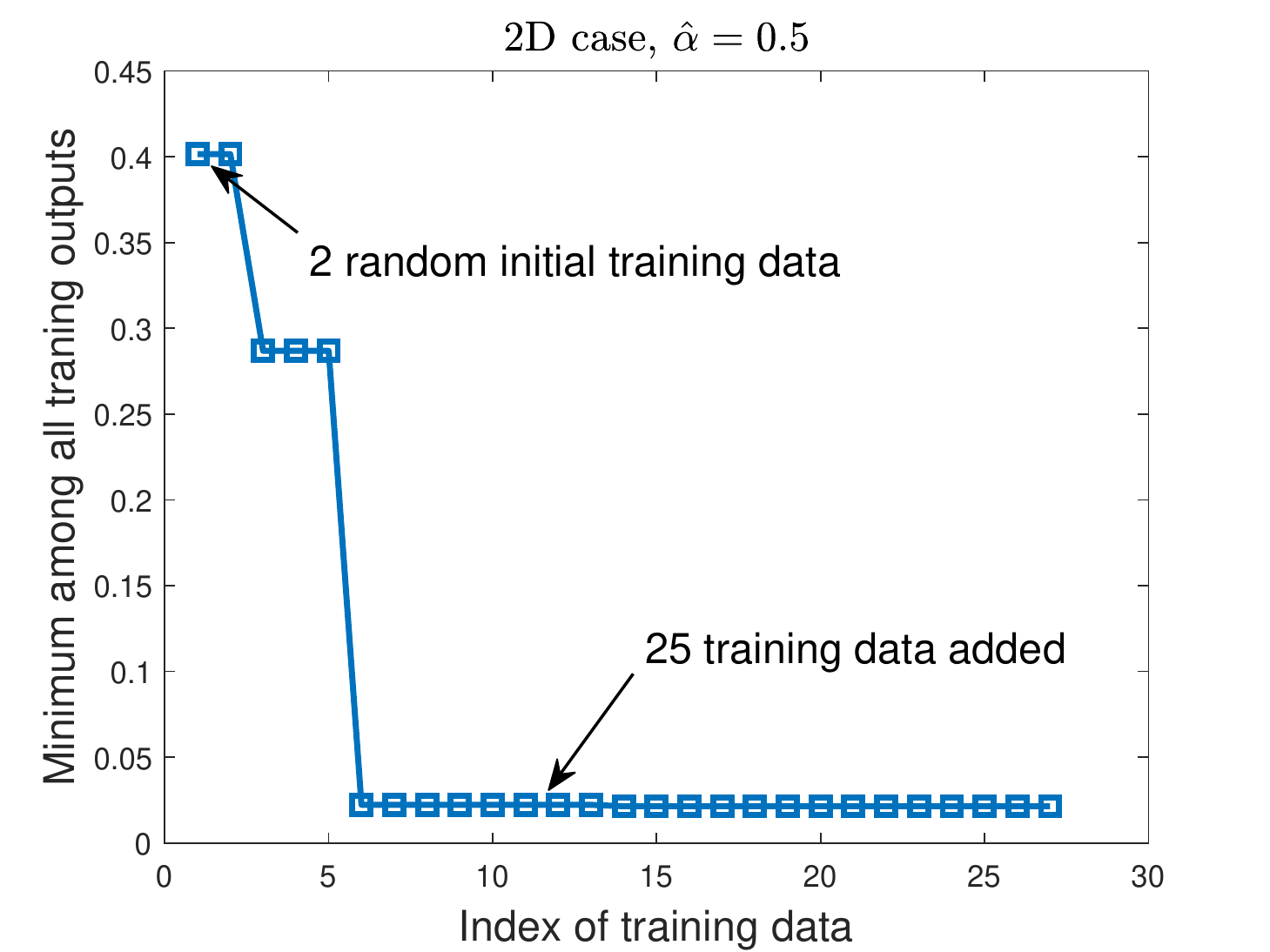}\end{center}
\end{minipage}
\begin{minipage}{0.46\linewidth}
\begin{center}
\includegraphics[scale=0.50,angle=0]{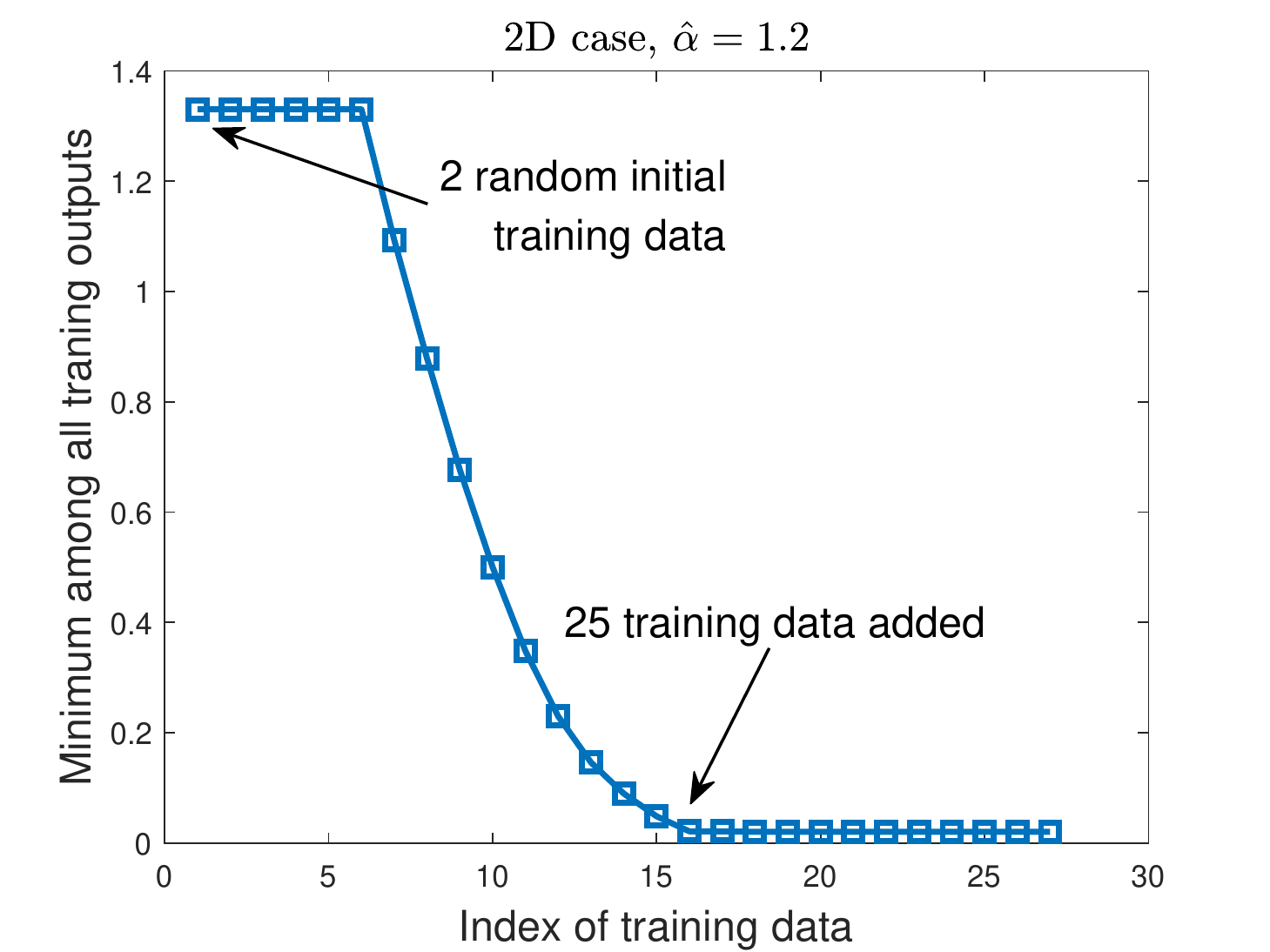}\end{center}
\end{minipage}
\caption{Convergence of the optimization problem for the two-dimensional case: Left: $\hat{\alpha} = 0.5$, right: $\hat{\alpha} = 1.2$.
}\label{fig:convergence2Ds0512}
\end{center}
\end{figure}

\section{Summary and discussion}\label{sec:5}
We presented a comparative study of nonlocal flocking dynamics using both the agent-based model and the continuum Eulerian model. Because animals in flocks generally do not interact mechanically and can be influenced by other individuals a certain distance away, we introduced nonlocal influence functions to consider the effects of animal communication in flocking dynamics. In particular, the microscopic dynamics of each individual in flocks is described by a Cucker-Smale particle model with nonlocal interaction terms, while the evolution of macroscopic quantities, i.e., mean velocity and population density, is modeled by the fractional partial differential equations (fPDEs). We performed agent-based simulations to generate the particle trajectories of each individual in flocking dynamics, and also solved the Euler equations  with nonlocal influence functions using a finite volume scheme. In one- and two-dimensional benchmarks of nonlocal flocking dynamics, we demonstrated that, given specified influence functions, the Euler system of equations  is able to capture the correct evolution of macroscopic quantities consistent with the collective behavior of the agent-based model.

Because experiments on flocking dynamics can get time series of trajectories of individuals in flocks using digital imaging or high-resolution GPS devices, we used the trajectories generated by the agent-based simulations to mimic the field data of tracking logs that can be obtained experimentally. Subsequently, we proposed a learning framework to connect the discrete agent-based model to the continuum fPDEs for nonlocal flocking dynamics. Instead of specifying a phenomenological fPDE with an empirical fractional order, we learned the effective non-local influence function in fPDEs directly from particle trajectories generated by the agent-based simulations. More specifically, we employed a Gaussian process regression (GPR) model implemented with the Bayesian optimization to learn the fractional order of the influence function from the particle trajectories. We showed in both one- and two-dimensional examples that the numerical solution of the learned Euler system of equations  solved by the finite volume scheme, can yield correct density distributions consistent with the collective behaviors of the discrete agent-based system. The relative error between the agent-based system and the learned Euler system of equations  is about $1\%$ for one-dimensional cases and $2\%$ for two-dimensional cases.

Although we only demonstrated the effectiveness of the proposed learning framework in relatively simple cases, i.e., one- and two-dimensional nonlocal flocking dynamics, this method established a direct connection between the discrete agent-based models to the continuum-based PDE models, and could serve as a paradigm on extracting effective governing
equations for nonlocal flocking dynamics directly from particle trajectories. It is worth noting that the agent-based model we used in the present work does not consider stochastic terms, and thus the training data of particle trajectories do not contain noise. However, the experimental data of tracking logs obtained by digital imaging may include noise from measurement uncertainty, where a multi-fidelity framework proposed by Babaee et al.~\cite{2018Babaee} can be used to handle different sources of uncertainties in the learning process. Moreover, in addition to the GPR-based learning method for connecting individual behavior to collective dynamics, it is also of interest to introduce deep learning strategies such as CNN (convolutional neural network)-based method~\cite{2018Pu} and the particle swarm optimization algorithm~\cite{2018Wang} to bridge the gap between flocking theory/modeling and experiments.

\section*{Acknowledgements}
This work was supported by the OSD/ARO/MURI on ``Fractional PDEs for Conservation Laws and Beyond: Theory, Numerics and Applications (W911NF-15-1-0562)".

\end{document}